\documentclass{amsart}
\usepackage[utf8]{inputenc}

\usepackage{amssymb,comment}
\usepackage{mathrsfs}
\usepackage{stmaryrd}
\usepackage{mathabx}
\usepackage{amsthm}

\usepackage{tikz-cd}

\usepackage{lmodern}
\usepackage{csquotes}

\usepackage[T1]{fontenc}

\definecolor{darkgreen}{rgb}{0,0.5,0}
\definecolor{darkblue}{rgb}{0,0,0.8}
\definecolor{darkred}{rgb}{0.8,0,0}
\definecolor{lightblue}{rgb}{0,0.6,0.8}

\usepackage[pdfencoding=auto,colorlinks,citecolor=darkgreen,linkcolor=darkblue,urlcolor=darkred]{hyperref}

\usepackage[english]{babel}
\usepackage{url}
\usepackage[
backend=bibtex,
style=alphabetic,
sorting=nyt,
backref=false,
maxitems=10,
minitems=10,
maxbibnames=10,
minbibnames=10,
maxcitenames=10,
mincitenames=10,
giveninits=true,
maxalphanames=7,
minalphanames=7,
doi=false,
isbn=false,
url=false
]{biblatex}
\bibliography{references}

\usepackage{enumitem}
\usepackage{colonequals}

\DeclareFontFamily{U}{wncy}{}
\DeclareFontShape{U}{wncy}{m}{n}{<->wncyr10}{}
\DeclareSymbolFont{mcy}{U}{wncy}{m}{n}
\DeclareMathSymbol{\Sha}{\mathord}{mcy}{"58}

\usepackage{microtype}

\theoremstyle{plain}
\newtheorem{theorem}{Theorem}[subsection]
\newtheorem{lemma}[theorem]{Lemma}

\newtheorem{proposition}[theorem]{Proposition}
\newtheorem{conjecture}[theorem]{Conjecture}

\newtheorem{remark}[theorem]{Remark}

\theoremstyle{definition}
\newtheorem{definition}[theorem]{Definition}

\newtheorem{assumption}[theorem]{Assumption}

\theoremstyle{remark}

\usepackage{cleveref}
\crefname{theorem}{Theorem}{Theorems}
\crefname{lemma}{Lemma}{Lemmata}
\crefname{corollary}{Corollary}{Corollaries}
\crefname{proposition}{Proposition}{Propositions}
\crefname{definition}{Definition}{Definitions}
\crefname{conjecture}{Conjecture}{Conjectures}
\crefname{question}{Question}{Questions}
\crefname{example}{Example}{Examples}
\crefname{algorithm}{Algorithm}{Algorithms}
\crefname{remark}{Remark}{Remarks}
\crefname{assumption}{Assumption}{Assumptions}

\def\ol#1{\overline{#1}}

\def\Alphabet{A,B,C,D,E,F,G,H,I,J,K,L,M,N,O,P,Q,R,S,T,U,V,W,X,Y,Z}
\def\alphabet{a,b,c,d,e,f,g,h,i,j,k,l,m,n,o,p,q,r,s,t,u,v,w,x,y,z}
\def\endpiece{xxx}
\def\makeAlphabet[#1]{\expandafter\makeA#1,xxx,}
\def\makealphabet[#1]{\expandafter\makea#1,xxx,}
\def\makeA#1,{\def\temp{#1}\ifx\temp\endpiece\else%
	\mkbb{#1}\mkfrak{#1}\mkbf{#1}\mkcal{#1}\mkscr{#1}\mkbs{#1}\expandafter\makeA\fi}%
\def\makea#1,{\def\temp{#1}\ifx\temp\endpiece\else\mkfrak{#1}\mkbf{#1}\mkbs{#1}\expandafter\makea\fi}%
\def\mkbb#1{\expandafter\def\csname bb#1\endcsname{\mathbb{#1}}}
\def\mkfrak#1{\expandafter\def\csname fr#1\endcsname{\mathfrak{#1}}}
\def\mkbf#1{\expandafter\def\csname b#1\endcsname{\mathbf{#1}}}
\def\mkcal#1{\expandafter\def\csname c#1\endcsname{\mathcal{#1}}}
\def\mkscr#1{\expandafter\def\csname s#1\endcsname{\mathscr{#1}}}
\def\mkbs#1{\expandafter\def\csname bs#1\endcsname{{\boldsymbol{#1}}}}
\def\makeop[#1]{\xmakeop#1,xxx,}
\def\mkop#1{\expandafter\def\csname #1\endcsname{{\mathrm{#1}}}} %
\def\xmakeop#1,{\def\temp{#1}\ifx\temp\endpiece\else\mkop{#1}\expandafter\xmakeop\fi}%
\def\makeup[#1]{\xmakeup#1,xxx,}
\def\mkup#1{\expandafter\def\csname #1\endcsname{{\mathrm{#1}\,}}} %
\def\xmakeup#1,{\def\temp{#1}\ifx\temp\endpiece\else\mkup{#1}\expandafter\xmakeup\fi}%
\makeAlphabet[\Alphabet]
\makealphabet[\alphabet]
\makeop[AJ,CH,cores,cris,dR,Hom,Tor,Ext,Fil,holim,hocolim,dim,Sel,GL,SL,H,ord,Sym,Gal,Ann,ind,Aut,End,cd,Frob,Gr,sp,Tot,sm,an,Pic,NS,Br,tors,Reg,sgn,ss,nr,tr,new,cts,ur,alg,pd,disc,cyc,rk,cork,ab,nr,tors,Iw,res,inf,Cl,Quot,corank,Heeg,BK,AJ,div,MSD,Tam,loc,Nek,MSD,ac,BDP,PD,MTT,sgn]
\makeup[Spec,Proj,Spwf,Sp,Sh,Spf,Sch,id,dR,rig,cris,im,coker]

\newcommand{\F}{\mathbf{F}}

\newcommand{\Q}{\mathbf{Q}}
\newcommand{\Qbar}{\ol\Q}
\newcommand{\Z}{\mathbf{Z}}
\newcommand{\fX}{\mathfrak{X}}
\newcommand{\D}{\mathbf{D}}
\newcommand{\fp}{\mathfrak{p}}

\newcommand{\eps}{\varepsilon}
\renewcommand{\epsilon}{\varepsilon}
\renewcommand{\theta}{\vartheta}
\renewcommand{\phi}{\varphi}

\newcommand{\mathup}[1]{\text{\textup{#1}}}
\renewcommand{\H} {\ensuremath{\mathup{H}}}

\newcommand{\GalQ}{G_\Q}
\newcommand{\Char}{\operatorname{Char}}

\newcommand{\defn}{\colonequals}
\newcommand{\defeq}{\colonequals}

\newcommand{\iso}{\simeq}
\newcommand{\isom}{\cong}

\newcommand{\inj}{\hookrightarrow}

\numberwithin{equation}{section}

\begin{document}
\title[TNC for newforms at Eisenstein primes]{On the Tamagawa number conjecture for newforms at Eisenstein primes}
\author{Mulun Yin}
\address{(M. Yin) Morningside Center of Mathematics, No.55, Zhongguancun East Road, Beijing, 100190, China}
\email{mulunyin@amss.ac.cn}
\date{\today}

\begin{abstract}
    We extend the results of~\cite{CGLS} to higher weight modular forms and prove a rank $0$ Tamagawa Number Conjecture formula (also known as the Bloch--Kato conjecture) for modular forms at good Eisenstein primes, under some technical assumption on periods. Under standard hypotheses (i.e. the injectivity of the $p$-adic Abel-Jabobi map and the non-degeneracy of the Gillet--Soul\'e height pairing), we also discuss some partial results towards a rank $1$ result. A conditional higher weight $p$-converse theorem to Gross--Zagier--Zhang--Kolyvagin--Nekovář is also obtained as a consequence of the anticyclotomic Iwasawa Main Conjectures.
\end{abstract}

\maketitle

\tableofcontents

\section*{Introduction}
\renewcommand{\thetheorem}{\Alph{theorem}}

\subsection{Background}In~\cite{CGLS}, Castella, Grossi, Lee and Skinner proved two anticyclotomic Iwasawa main conjectures for elliptic curves at good ordinary Eisenstein primes over an imaginary quadratic field under certain hypotheses. Together with a rank $0$ BSD formula obtained by Greenberg and Vatsal in~\cite{GV00}, they proved a rank $1$ BSD formula for elliptic curves over $\bQ$ in the residually reducible setting. Most of their results have since been generalized to higher weight modular forms in~\cite{KY24}, and several hypotheses have been removed in \textit{op. cit.} (e.g. the condition that $\phi|_{G_p},\psi|_{G_p}\ne \mathbf{1},\omega$ where $\phi,\psi$ are the characters appearing in the semisimplification of $E[p]$, $\omega$ is the mod-$p$ cyclotomic character and $G_p$ denotes the decomposition group at $p$) as well as in~\cite{CGS} (e.g. assumptions on the characters coming from~\cite{GV00}). 

In~\cite{KY24}, the Iwasawa Main Conjectures in~\cite{CGLS} have been extended to higher weight modular forms, but a BSD formula is only obtained for elliptic curves. In this paper, we will further extend this application to modular forms of certain weights greater than $2$ and prove the Tamagawa Number Conjecture in rank $0$ as well as a higher weight $p$-converse theorem, under technical assumptions. One immediately notes that in the higher weight setting, there are multiple $L$-values of interest, namely those at $1,2,...,r$ where $k=2r$ is the weight. The rank $0$ results concerning all the values $L(f,s)$ with a strict weight range will be obtained conditionally on certain cyclotomic Iwasawa Main Conjectures for $f$, while rank $1$ results about the central value $L'(f,r)$ can be further obtained from anticyclotomic Iwasawa theory, provided a Gross--Zagier type formula for generalized Heegner cycles over $\Gamma_1(N)$ becomes available. Unfortunately, we won't be able to say anything about other values of the derivative since Zhang's formula only concerns $L'(f,r)$.

\subsection{The main results}\label{results}
Let $f\in S_{2r}^{new}(\Gamma_0(N))$ be a newform of weight $2r\geq 2$ and level $N$ with trivial Nebentypus. Let $\bQ(f)$ be the coefficient field of $f$, i.e., the finite extension of $\bQ$ generated by the Fourier coefficients $\{a_n(f)\}_{n\geq 1}$ of $f$, with ring of integers $\bZ(f)$. Fix an odd prime $p\nmid N$ such that $a_p$ is a $p$-adic unit. Equivalently, this means $p$ is a prime of good ordinary reduction for $f$. Let $F$ be a finite extension of the completion of $\bQ(f)$ at a chosen place $\fp$ above $p$ with ring of integers $\cO$, and denote by
\[
\rho_f(1-r): \Gal(\Qbar/\Q) \to  \GL_2(F)
\]
the self-dual Tate twist of the $p$-adic Galois representation $\rho_f$ attached to $f$ (dual to Deligne's construction). We denote from now by $V_f$ the self-dual representation attached to $f$ as above. When necessary, we will also study $\rho_f$ for the representation before self-dual twist. Let $\F$ be the residual field of $F$. We assume $p$ is an Eisenstein prime for $f$, meaning that the residual representation $\ol\rho_f$ is reducible. Then the self-dual twist gives rise to a decomposition
\[\ol\rho_f^\ss(1-r) \isom \F((\epsilon\omega^{r-1})\omega) \oplus \F((\epsilon\omega^{r-1})^{-1}).\] We denote characters occurring in $\ol\rho_f^\ss(1-r)$ by $\phi$ and $\psi$, so $\phi\psi=\omega$.

Let us briefly explain the ideas that go into the proof of our first two main results~\cref{pConvintro} and~\cref{rank0}. Recall that we only consider Eisenstein primes for $f$, meaning that the semisimplification of the residual representation \[\ol\rho_f^{ss}=\F(\phi)\oplus\F(\psi)\] is reducible. Equivalently, there is an extension \begin{equation}\label{chars}
	0\to\F(\phi)\to\ol\rho_f\to\F(\psi)\to 0.\end{equation}
By an abuse of notation, we assume $\ol\rho_f$ is self-dual so that $\phi\psi=\omega$.
However, to make sense of the residual representation, one needs to make a choice of a Galois stable lattice $T_f$ in $V_f$. Unlike in the residually irreducible case, such choice is not unique up to homothety. Fortunately, both the Iwasawa main conjecture and the BSD conjecture (see~\cite{KY24}) are invariant under isogeny, so we are free to choose any lattice (by Ribet's lemma, this amounts to a choice of the ordering of the characters appearing in~\eqref{chars}). Actually, a choice of the lattice plays a crucial role in the proof of their main results. For our application, it is sufficient to choose the `canonical' lattice (see~\cref{pAJ/N}).

A key input in this paper is an anticyclotomic Iwasawa main conjecture for modular forms proved in~\cite{KY24}. Let $K$ be an imaginary quadratic field and $\Gamma_K=\Gal(K_\infty/K)$ be the Galois group of the anticyclotomic $\bZ_p$-extension of $K$. Let $\Lambda_K=\cO\llbracket\Gamma_K\rrbracket$ be the Iwasawa algebra and let $\Lambda_K^{\nr}\coloneq\Lambda_K\hat{\otimes}_{\bZ_p}\bZ_p^{\nr}$, for $\bZ_p^{\nr}$ the completion of the ring of integers of the maximal unramified extension of $\bQ_p$. The \textit{(Greenberg's) Iwasawa Main Conjecture} we need takes on the following form:
\begin{conjecture}\label[conjecture]{imcintro}
	Let $f\in S^{new}_{2r}(\Gamma_0(N))$ be a newform of weight $2r\geq 2$ and $p\nmid 2N$ be an Eisenstein prime of good ordinary reduction for $f$. If $K$ is an imaginary quadratic field satisfying the Heegner hypothesis where $p$ splits, then $\fX_f$ is $\Lambda_K$-cotorsion, and
	\[\Char(\fX_f)\Lambda_K^{\nr}=(\cL_f^\BDP)\] as ideals in $\Lambda_K^{\nr}$
\end{conjecture}
Here the left hand side is the characteristic ideal of a certain Selmer group $\fX_f$ for $f$ and the right hand side is an associated $p$-adic $L$-function which lives in $\Lambda_K^{\nr}$. The Heegner hypothesis states that every prime dividing $N$ is split in $K$. 

Under~\cref{assum}, ~\cref{imcintro} is now a theorem in~\cite{KY24} if $r$ is odd. In fact, it is further shown that the above conjecture is equivalent to a \textit{Heegner Point Main Conjecture}.

\begin{conjecture}\label{HPMC}
	Let $f\in S^{new}_{2r}(\Gamma_0(N))$ be a newform of weight $2r\geq 2$ and $p\nmid 2N$ be an Eisenstein prime of good ordinary reduction for $f$, and let $K$ be an imaginary quadratic field satisfying the Heegner hypothesis where $p$ splits. Then both $S$ and $X$ have $\Lambda_K$-rank one, and\[
	\Char_{\Lambda_K}(X_\tors)=\Char_{\Lambda_K}(S/\Lambda_K\cdot\kappa^\Heeg)^2,
	\]
	where $X_\tors$ denote the $\Lambda$-torsion submodule of $X$.
\end{conjecture}

Here $S=\H^1_{\cF_{\Lambda_K}}(K,\bT)$ and $X=\H^1_{\cF_{\Lambda_K}}(K,M_f)^\vee$ are (Pontryagin duals of) certain $\Lambda_K$-adic Selmer groups introduced in~\cite[section 3]{CGLS}, and $\kappa^\Heeg$ is a $\Lambda_K$-adic \textit{Heegner cycle} conjectured to be non-torsion. Again under~\cref{assum},~\cref{HPMC} is now a theorem when $r$ is odd.

A standard consequence of a Heegner Point Main Conjecture is a $p$-converse theorem of Gross--Zagier--Kolyvagin (see for example~\cite{KY24b}). For higher weight modular forms, the Gross--Zagier formula is extended by Zhang and the theorem of Kolyvagin is extended by Nekovář. It should be noted that the \textit{classical} Heegner cycles in Zhang's formula as well as in Nekovář's work are different from the \textit{generalized} Heegner cycles appearing in the above Heegner Point Main Conjecture. Their difference is understood well enough to yield the $p$-converse theorem, yet a complete proof of the rank $1$ Tamagawa Number formula remains mysterious. 

Let $L(f,s)$ be the $L$-function of the Galois representation $\rho_f$. The \textit{analytic rank} of $f$ is it's order of vanishing at $s=r$ (which agrees with the order of vanishing of the $L$-function $L(V_f,s)$ at $s=1$ after self-dual twist). Let $A_f$ be defined by the following short exact sequence\[
0\to T_f\to V_f\to A_f\to 0,\]
and let $\H^1_\BK(K,A_f)$ be the Bloch--Kato Selmer group for $A_f$ over $K$. Let \[\AJ^f_K:\CH^{r/2}(\tilde{\cE}^{2r-2}(N)/K)_0\otimes \cO\to\H^1_{\cts}(K,A_f)\]
be the $p$-adic Abel--Jacobi map attached to the Kuga--Sato variety $\tilde{\cE}^{2r-2}(N)$ of level $N$ and weight $2r$ (see~\cref{2} for more discussion about this map). Then by work of Nekovář (see for example~\cite[Proposition 11.2]{Nek92}), there is an exact sequence\begin{equation}\label{AJdescent}
0\to \im(\AJ^f_K)\otimes \bQ_p/\bZ_p \to \H^1_\BK(K,A_f)\to \Sha_\Nek(f/K)\to 0
\end{equation}
which defines $\Sha_\Nek(f/K)$, necessarily a $p$-primary group. The \textit{algebraic rank} of $f$ is the $\bZ_p$-corank of the first group. We mention that $\Sha_\Nek(f/K)$ is generally different from the Tate--Shafarevich group defined a l\`a Bloch--Kato as the quotient $\Sha_\BK(f/K)\defeq\H^1_\BK(K,A_f)/\H^1_\BK(K,A_f)_{\div}$ (see~\cite[Remark 8.2]{Mas}). However note that $\im(\AJ^f_K)\otimes \bQ_p/\bZ_p$ maps injectively into $\H^1_\BK(K,A_f)_\div$, and hence there is a surjection $\Sha_\Nek(f/K)\twoheadrightarrow \Sha_\BK(f/K)$, the latter always being finite whenever $\H^1_\BK(K,A_f)$ has finite $\cO$-corank. When the $\frp^\infty$-part of the classical Tate--Shafarevich group $\Sha(A_f/K)[\frp^\infty]=\Sha(A_f/K)[p^\infty]\otimes_{\bZ(f)}\cO$ is finite, there is an equality $\Sha_\BK(f/K)=\Sha(A_f/K)[\frp^\infty]$.

Zhang's Gross--Zagier formula~\cite{Zhang1997} and Nekovář's theorem~\cite{Nek92} generalize Gross--Zagier--Kolyvagin theorem to modular forms. Namely, whenever the analytic rank is less than $1$, it's equal to the algebraic rank. The higher weight $p$-converse theorem is the first main result of this paper. See~\cref{pconv}.

\begin{theorem}\label{pConvintro}
		Let $f\in S^{new}_{2r}(\Gamma_0(N))$ be a newform of weight $2r\geq 2$ where $r$ is odd and $p\nmid 2N$ be an Eisenstein prime of good ordinary reduction for $f$. Assume the Gillet--Soul\'e pairing is non-degenerate and all Abel--Jacobi maps are injective. Let $t\in\{0,1\}$. Then\[
		\corank_{\bZ_p}(\H^1_{\BK}(\bQ,A_f))=t\Rightarrow \ord_{s=r}L(f/\bQ,s)=t,
		\]
		and so $\dim_F(\im(\AJ^f_\bQ)\otimes\bQ)=t$ and $\#\Sha_\Nek(f/\bQ)[p^\infty]<\infty$.
\end{theorem}

Another consequence of the anticyclotomic Iwasawa Main Conjectures over $K$, by the descent arguments of~\cite{CGS}, is a cyclotomic Iwasawa Main Conjecture over $\bQ$ as below. Let $\fX(f)\defeq \H^1_\Gr(\bQ,M_f')^\vee$ be the Pontryagin dual of the $p$-primary Selmer group for $f$ over $\bQ$ and let $\cL_f^\MSD$ be the Mazur--Swinnerton-Dyer $p$-adic $L$-function for $f$ (see\cref{cycIwa} for definitions). Here $\Lambda_\bQ\defeq\bZ_p \llbracket\Gal(\bQ_\infty/\Q)\rrbracket$ is the cyclotomic Iwasawa algebra over $\bQ$.

\begin{conjecture}
	Let $f\in S_{2r}^{new}(\Gamma_0(N))$ be a newform. Let $p\nmid N$ be an Eisenstein prime for $f$, i.e., $\ol\rho_f$ is reducible. Then 
	\[\Char_{\Lambda_\bQ}(\H^1_\Gr(\bQ,M_f')^\vee)=(\cL_f^{\MSD})\]
	as ideals in $\Lambda_\bQ$.
\end{conjecture}	
 This conjecture was first proved in~\cite{GV00} for elliptic curves under technical assumptions that force the Iwasawa $\mu$-invariants of both sides to vanish. Later it was extended in~\cite{CGS} based on~\cite{CGLS} to allow positive $\mu$-invariants using new ideas, with weaker technical conditions, again for elliptic curves. It is now an unconditional theorem for elliptic curves by the modification of~\cite{CGLS} in~\cite{KY24}. A generalization of the results in ~\cite{CGS} to higher weight modular forms would yield an unconditional proof of this conjecture. However, currently only a direct generalization of~\cite{GV00} is available. If one in addition assumes that the modular form has weight $2r\leq p-1$, the above conjecture is partially proved in~\cite{Hir18} under the assumptions in~\cite{GV00}. See~\cref{cycMC}. However, the IMC proved in~\cite{Hir18} is not the right one we need due to `branch issues' (see~\cref{cycIwa}). One aim of this paper is to derive an appropriate generalization of~\cite{GV00} based on~\cite{Hir18} that is good for a rank $0$ Tamagawa Number Conjecture, under some technical assumptions.

If we assume that $\phi|_{G_p},\psi|_{G_p}\ne \mathbf{1},\omega$, then we can apply the control theorems for higher weight modular forms in~\cite{LV21}. Note that a consequence of this additional hypothesis is that $\H^0(\bQ,A_f)=0$. Similar to~\eqref{AJdescent}, there is a short exact sequence\begin{equation}\label{AJdescentQ}
	0\to \im(AJ^f_\bQ)\otimes \bQ_p/\bZ_p\to \H^1_\BK(\bQ,A_f)\to \Sha_\Nek(f/\bQ) \to 0
	\end{equation}
that defines $\Sha_\Nek(f/\bQ)$. As a corollary of the cyclotomic Iwasawa Main Conjecture and the cyclotomic control theorem, we obtain the $p$-part Tamagawa Number formula in the rank $0$ case. See~\cref{pTNC}.

\begin{theorem}[$p$-part Tamagawa Number Conjecture in rank $0$]\label{rank0}
	Let $f\in S_{2r}^{new}(\Gamma_0(N))$ be a newform of weight $2r\leq p-1$ where $p\nmid 2N$ is an Eisenstein prime for $f$, i.e., $\ol\rho_f$ is reducible, and $r$ is odd. Assume that the sub-representation $\F(\phi)$ of $\ol\rho_f$ is ramified at $p$ and even when restricted to the decomposition group. Further assume that $\ord_p(\frac{\Omega^+_f}{\Omega^+_g})\geq 0$, where $g$ is the eigenform obtained from $f$ by removing all Euler factors at primes dividing $N$. Assume $L(f,r)\ne 0$. We have
	\[\ord_p(\frac{L(f,r)}{\Omega_f})=\ord_p(\#\Sha_\Nek(f/\bQ)\#\Tam(A_f/\bQ))\]
	Here $\Omega_f=\Omega^+_f$ (and similarly $\Omega^+_g$) is the period attached to $f$ as in~\cref{cycpL}.
\end{theorem}

We remark that $\frac{\Omega^+_f}{\Omega^+_g}$ is conjectured to be a $p$-adic unit, and it is known in the case of elliptic curves (see~\cite[Lemma 3.6]{GV00}).

We also prove an anticyclotomic control theorem that is good for a rank $1$ Tamagawa Number formula. However, due to the lack of a Gross--Zagier--Zhang type formula and insufficient understanding of the relation between the generalized Heegner cycles and the $L$-function, we will not try to prove the rank $1$ formula here.

Under the hypotheses in~\cref{control}, one has the following.

\begin{theorem}[Anticyclotomic Control Theorem]\label{acon}
		Let $f^\Sigma_{ac}$ be a generator of the characteristic ideal $\Char_{\Lambda_K}(X^\Sigma_{ac}(M_f))$ of the torsion $\Lambda_K$-module $X^\Sigma_{ac}(M_f)$, then\begin{equation*}\label{Antcon}
		\#\cO/f^\Sigma_{ac}(0)=\frac{\#\Sha_\BK(f/K)\cdot C^\Sigma(A_f)}{(\#\H^0(K,A_f))^2}(\#\delta_v)^2,
	\end{equation*}
where $C^\Sigma(A_f)$ and $\delta_v$ are in~\cref{control} accounting for Tamagawa numbers, local and global index depending on a choice of a Heegner cycle.
\end{theorem}

Finally, we also do some computations towards the verification of a rank $1$ Tamagawa Number formula in~\cref{rank1}.

\subsection{Methods of proof and outline of the paper}
As this work is a direct generalization of the results in~\cite{CGLS}, the main ideas of proofs are similar. However, as the situations are more mysterious for higher weight modular forms than for elliptic curves and less is known, we often need to first establish analogues of existing results, or do some extra verification and comparison that are unnecessary in weight $2$.

In~\cref{iwasawa}, we review some background knowledge from Iwasawa theory. In particular, we discuss the anticyclotomic Iwasawa Main Conjectures proved in~\cite{KY24} and a cyclotomic Iwasawa Main Conjecture in the style of~\cite{GV00} for higher weight modular forms, for a different `branch' than those studies in the existing literature (e.g.,~\cite{SU14} and~\cite{Hir18}). In particular, we also prove Kato's integral divisibility for modular forms at Eisenstein primes under some conditions. These will be the foundation for our proofs of the Tagamawa Number formulas. We also discuss how to use the Main Conjectures before self-dual twist to study the Tamagawa Number Conjectures after self-dual twist.

In~\cref{2}, we review the constructions of Heegner cycles and $p$-adic Abel--Jacobi maps in two different settings. The first setting is based on Kuga--Sato varieties over the modular curve $X_1(N)$, where we have the theory of Bertolini--Darmon--Prasana that relates the Abel--Jacobi image of Heegner cycles over $\Gamma_1(N)$ to their $p$-adic $L$-functions. This provides a $p$-adic Gross--Zagier formula we will need to study a rank $1$ Tamagawa Number formula. On the other hand, there is a second setting where everything is defined over the congruence subgroup $\Gamma(N)$. The combined work of Zhang~\cite{Zhang1997} and Nekovář~\cite{Nek92} provides a generalization of the classical Gross--Zagier--Kolyvagin's theorem to higher weight modular forms, where the Heegner points for elliptic curves are replaced by the Heegner cycles over $\Gamma(N)$, while no analogue is known for $\Gamma_1(N)$. As far as the Tamagawa Number conjecture is concerned, it should not matter which kind of Heegner cycles we choose, as they only show up in intermidiate steps. However, due to the above asymmetric situation, we do not have enough tools to unite them. Nevertheless, we mention a few comparison results of Thackery~\cite{Thackeray2022} that relate the indices of the Heegner cycles in the Abel--Jacobi images that is good enough for a $p$-converse theorem.

In~\cref{section:control}, we discuss three control theorems with exact formulae for modular forms, one cyclotomic and two anticyclotomic (one of Greenberg type and one of Jetchev--Skinner--Wan type, see \cref{acon}). The cyclotomic control theorem over $\bQ$ is good for the rank $0$ Tamagawa Number Conjecture and is the main result of~\cite{LV21}, while the anticyclotomic control theorem over $K$ of Jetchev--Skinner--Wan type is needed for the rank $1$ result and is discussed in~\cite{Thackeray2022} in the irreducible setting. We explain how to adapt these results to the residually reducible case. On the other hand, the control theorem of Greenberg type will be needed for the $p$-converse theorem. Since we do not have access to a general cyclotomic main conjecture, it is sufficient to consider the cyclotomic control theorem where there is no global torsion, i.e. when $\H^0(\bQ_p,A_f)=0$. However, the anticyclotomic control theorems do allow non-trivial global torsion.

In~\cref{main}, we first prove the rank $0$ Tagamawa Number formula (\cref{rank0}) under a technical assumption on periods of certain modular forms by combining the Main Conjectures from~\cref{iwasawa} and the control theorem from~\cref{section:control}. This result is also necessary for a rank $1$ formula. Then we prove a $p$-converse theorem (\cref{pConvintro}) to the theorem of Gross--Zagier--Zhang--Kolyvagin--Nekovář. In doing so, we need to choose some auxiliary imaginary quadratic fields $K$ where the Iwasawa Main Conjectures hold, and we also need to compare different Heegner cycles to related the global $L$-functions to the $p$-adic ones. Finally, we compute some specific $p$-indices appearing in the anticyclotomic control theorem using Fontaine--Laffaille theory, and study the compatibility of all computational results towards the rank $1$ Tamagawa Number formula.

\subsection{Relation to previous works}
Our main results are based on the Iwasawa Main Conjectures studie by several authors. The cyclotomic Iwasawa Main Conjecture for elliptic curves in the good Eisenstein case was first proved in~\cite{GV00} and generalized for some higher weight modular forms in~\cite{Hir18}. The anticyclotomic Main Conjectures for elliptic curves in the good Eisenstein case was first proved in~\cite{CGLS} and generalized for some higher weight modular forms in~\cite{KY24}.

Prior to this work, control theorems for modular forms have been studied in~\cite{LV21} (cyclotomic) and~\cite{JSW2017} (anticyclotomic). We explain how to adapt them in the Eisenstein case.

A $p$-converse theorems for modular forms in the irreducible setting was given in~\cite{LV23}. Their approach is based on Kolyvagin's Conjecture and is different from ours.

Tamagawa Number conjecture formula in rank $0$ and $1$ for motives of modular forms are discussed in~\cite{LV23}. A higher weight BSD formula in rank $1$ is also obtained in~\cite{Thackeray2022}. All these results are in the residually irreducible setting.

\subsection{Notations}
For any subextension $L/\Q$ of $\bar{\bQ}/\bQ$, we let $G_L\coloneq \Gal(\bar{\bQ}/L)$ denote its absolute Galois group. For a cohomology group $\H^i(\cdot,-)$, we write $\H^i(L,-)$ in place of $\H^i(G_L,-)$.

\subsection{Future Work}
A generalization of~\cite{CGS} to higher weight modular forms would hopefully help remove the technical assumptions in~\cite{GV00} and~\cite{Hir18}. In particular, elliptic curves with non-trivial torsion groups will be covered and therefore we hope to extend the cyclotomic control theorem from~\cite{LV21} to allow torsion as well. It would also be interesting to study the canonical periods for related modular forms so that one might obtain an unconditional proof of the rank $0$ $p$-part Tamagawa Number Conjectures by directly using the arguments in~\cite{Hir18}.

On the other hand, since a rank $1$ result on the $p$-part Tamagawa Number Conjecture is desired, we hope to examine a Gross--Zagier type formula for Heegner cycles defined over $\Gamma_1(N)$.

\subsection{Acknowledgment}
This is partially a part of the author's Ph. D. thesis. We thank his advisor Francesc Castella for his guidance.

\renewcommand{\thetheorem}{\arabic{section}.\arabic{subsection}.\arabic{theorem}}

\section{Iwasawa theory}\label{iwasawa}

To state the main results in~\cite{KY24} that yield a proof of~\cref{imcintro} in the introduction, we first recall the algebraic side and the analytic side of the anticyclotomic Iwasawa theory.

Let $K \subset \Qbar$ be an auxiliary imaginary quadratic field in which $p = v\ol{v}$ splits, with $v$ the prime of $K$ above $p$ induced by $\iota_p$. We also fix an embedding $\iota_\infty: \Qbar \inj \bC$.

Let $G_K = \Gal(\Qbar/K) \subset \GalQ \defn \Gal(\Qbar/\Q)$, and for each place $w$ of $K$ let $I_w \subset G_w \subset G_K$ be the corresponding inertia and decomposition groups. Let $\Frob_w \in G_w/I_w$ be the arithmetic Frobenius. For the prime $v \mid p$, we assume $G_v$ is chosen so that it is identified with $\Gal(\Qbar_p/\Q_p)$ via $\iota_p$.

Recall that $F/\bQ_p$ is a finite extension of $\bQ_p$ containing the Fourier coefficients of $f$. Let $\cO$ denote its ring of integers and $\F$ be its residue field. Denote by $\frp$ a prime ideal of $\cO$ lying above $p$. Let $\Gamma \defn \Gal(K_\infty/K)$ be the Galois group of the anticyclotomic $\Z_p$-extension $K_\infty$ of $K$, and let $\Lambda_K \defn \cO\llbracket\Gamma\rrbracket$ be the anticyclotomic Iwasawa algebra. We shall often identify $\Lambda_K$ with the power series ring $\cO\llbracket T\rrbracket$ by setting $T = \gamma - 1$ for a fixed topological generator $\gamma \in \Gamma$. 

 Throughout, we assume the Heegner hypothesis\begin{equation*}\tag{Heeg}\label{heeg}\text{every prime }l\text{ dividing }N\text{ splits in }K.\end{equation*}

\subsection{The algebraic side}
In this section, we define the Selmer groups and discuss some of its important properties. The goal is to describe the Iwasawa theoretic results on the algebraic side which will be compared to those on the analytic side in the next section.

Let $\Sigma\supset\{v,\bar{v},\infty\}$ be a finite set of places of $K$. We define the Selmer group with \textit{unramified local conditions} for $f$ as{\small\[
	\H^1_{\cF_{ur}}(K,M_f)=\ker\bigg(\H^1(K^\Sigma/K, M_f) \to \H^1(I_v,M_f)^{G_v/I_v}\times\prod_{w\mid l\ne p, l\in\Sigma}\H^1(I_w,M_f)^{G_w/Iw}\bigg)\]}
where $K^\Sigma$ is the maximal extension of $K$ unramified outside $\Sigma$ and $M_f=T_f\otimes \Lambda_K^\vee$
For a set $S\subset\Sigma\setminus\{v,\ol v,\infty\}$, we define the $S$-imprimitive Selmer group for $f$ as
{\small\[
	\H^1_{\cF^S_{ur}}(K,M_f)=\ker\bigg(\H^1(K^\Sigma/K, M_f) \to \H^1(I_v,M_f)^{G_v/I_v}\bigg)\]}
It is proved in~\cite{KY24} that these Selmer groups are $\Lambda_K$-cotorsion and the global-to-local maps defining the Selmer groups are surjective. Let $\fX_f$ and $\fX^S_f$ denote the Pontryagin dual of the primitive and imprimitive Selmer groups for $f$ respectively.

Replacing $M_f$ with $M_f[\frp]$ in the above definitions, we also get the (primitive and imprimitive) \textit{residual Selmer groups} for $f$.

For a character $\theta:G_K\to \F^\times$ whose conductor is only divisible by primes split in $K$, define $M_\theta\coloneq \cO(\theta)\otimes_\cO \Lambda_K^\vee$. Replacing $M_f$ with $M_\theta$ in the above definitions, we also contain the Selmer groups for $\theta$.

The imprimitive residual Selmer groups allow us to compare $f$ with the characters $\phi,\psi$ appearing in the semisimplification of $\ol\rho_f$. Using the known results about the characters, it is shown in~\cite{CGLS} that the Iwasawa $\mu$-invariants of the Selmer groups are vanishing and the $\lambda$-invariant of $f$ is related to those of the characters. For example, in the setting of \textit{loc. cit.} (i.e. when $f$ corresponds to an elliptic curve, with some technical conditions), one can show that for $?=f,\phi,\psi$, \[
\lambda(\fX_?)=\dim_{\F_p} \H^1_{\cF^S_\ur}(K,M_?[p])
\]and there is a short exact sequence\[
0\to \H^1_{\cF^S_\ur}(K,M_\phi[p]) \to \H^1_{\cF^S_\ur}(K,M_f[p])\to \H^1_{\cF^S_\ur}(K,M_\psi[p])\to 0,
\] so we have the following simple relation\[
\lambda(\fX^S_f)=\lambda(\fX^S_\phi)+\lambda(\fX^S_\psi).\]
In general cases, the above relation between the $\lambda$-invariants should be satisfied, even if the above sequences may not be exact, except in one case when one of the characters is the trivial character over $G_K$. This exceptional case is studied in~\cite{KY24}, and the difference is consistent with the Iwasawa main conjecture for the trivial character which differs from others.

We record the following theorem from \textit{op.\ cit.} which compares the algebraic Iwasawa invariants of $f$ to those of the characters.

\begin{theorem}\label{algmain}
	Let $\phi,\psi$ be the characters appearing in the semisimplification of $\ol\rho_f$. If neither of them is the trivial character on $G_K$, then the module $\frX_f$ is $\Lambda_K$-torsion with $\mu(\frX_f) = 0$ and
	\[
	\lambda(\frX_f) = \lambda(\frX_\phi) + \lambda(\frX_\psi) + \sum_{w \in S}\big\{\lambda(\cP_w(\phi)) + \lambda(\cP_w(\psi)) - \lambda(\cP_w(f))\big\}.
	\]
	When $\phi|_{G_K}$ or $\psi|_{G_K}=\mathbf{1}$, the same results hold except that the relation between $\Lambda_K$-invariants now becomes\[
	\lambda(\frX_f) +1 = \lambda(\frX_\phi) + \lambda(\frX_\psi) + \sum_{w \in S}\big\{\lambda(\cP_w(\phi)) + \lambda(\cP_w(\psi)) - \lambda(\cP_w(f))\big\}.
	\]
	
\end{theorem}
\begin{proof}
	This is a combination of~\cite[Section 1.5]{KY24} and~\cite[Section 1.5]{CGLS}. We mention that in the second case, the result holds no matter whether $\F(\mathbf{1})$ is a subrepresentation or a quotient representation of $\ol\rho_f$.
\end{proof}

\subsection{the analytic side}
In this section, we describe the BDP $p$-adic $L$-functions for $f$ and the Katz $p$-adic $L$-function for the characters obtained from~\cite[section 2.1]{KY24}, and discuss their useful properties. The following hypotheses are in effect throughout this section. They all come from~\cite{CGLS}.
\begin{assumption}\label[assumption]{assum}
	\begin{enumerate}
		\item $p=v\bar{v}$ is split in $K$
		\item The Heegner hypothesis
		\item The discriminant $D_K$ of $K$ is odd and $D_K\ne -3$ 
	\end{enumerate}
\end{assumption}

\subsubsection{The Bertolini--Darmon--Prasanna $p$-adic $L$-functions.}

The Heegner hypothesis allows one to fix an integral ideal $\frN \subset \cO_K$ with
\begin{equation*} \label{eq:frNHeegner}
    \cO_K/\frN \iso \Z/N.
\end{equation*}

\begin{proposition}[$p$-adic interpolation property]
	There exists an element $\cL_f^\BDP\in\Lambda_K^{ur}$ characterized by the following interpolation property: If $\hat{\xi} \in \frX_{p^\infty}$ is the $p$-adic avatar of a Hecke character $\xi$ of infinity type $(n,-n)$ with $n \geq 0$ and $p$-power conductor, then
		{\small\[
		\cL_f^\BDP(\hat{\xi}) = \frac{\Omega_p^{4n}}{\Omega_K^{4n}}\cdot \frac{4\Gamma(n+\frac{k}{2})\Gamma(n-\frac{k}{2}+1)\xi^{-1}(\frN^{-1})}{(2\pi)^{2n+1}(\sqrt{D_K})^{2n-1}}\cdot (1-a_p(f)p^{-r}\xi_{\ol \frp}(p)+\xi_{\ol \frp}(p^2)p^{-1})^2\cdot L(f/K,\xi,1).  
		\]}
\end{proposition}

This BDP $p$-adic $L$-function can also be explicitly constructed using results from~\cite{CastellaHsieh}. See~\cite[Theorem 2.1.1]{CGLS} for more detail in the elliptic curve case. It should be mentioned that in~\cite[Theorem 2.1.1]{KY24} they can pick $c=c_0=1$ because it is sufficient for their application to a BSD conjecture for elliptic curves, but we will assume that we are in a more general case where we can at most say $c=c_0$ is prime to $p$ (see~\cite[Assumption 5.12]{BDP13}). From~\cite[Theorem 3.8]{CastellaHsieh}, the above interpolation formula still holds (up to $p$-adic units) when $(c_0,p)=1$. 

Note that the above interpolation property characterizes the BDP $p$-adic $L$-function, but one cannot directly plug in the norm character $\bN_K^r$ of infinity type $(r,r)$ because it's outside the range of interpolation. The value at $\bN^r$ can be computed using the main theorem of~\cite{BDP13} (see~\cref{BDP}), and since $L_p(f,\bN^r)$ corresponds to the value $L(f/K,\bN^{-r},0)=L(f/K,\mathbf{1},r)=L(f/K,r)$, it is the constant term $\cL^\BDP_f(0)$.

\subsubsection{The Katz $p$-adic $L$-functions.}
\begin{proposition}
	There exists an element $\cL_\theta \in \Lambda_K^\nr$ characterized by the following interpolation property: For every character $\xi$ of $\Gamma$ crystalline at both $v$ and $\bar{v}$ and corresponding to a Hecke character of $K$ of infinity type $(n,-n)$ with $n\in\bZ_{>0}$ and $n\equiv 0$ (mod $p-1$), we have
	\begin{align*}
		\cL_{\theta}(\xi)=\frac{\Omega_p^{2n}}{\Omega_\infty^{2n}}\cdot 4\Gamma(n+\tfrac{k}{2})\cdot\frac{(2\pi i)^{n-\frac{k}{2}}}{\sqrt{D_K}^{n-\frac{k}{2}}}\cdot(1-\theta^{-1}(p)\xi^{-1}(v))\cdot(1-\theta(p)\xi(\bar{v})p^{-1})\\
		\times\prod_{\ell\mid C}(1-\theta(\ell)\xi(w)\ell^{-1})\cdot L(\theta_K\xi\mathbf{N}_K^\frac{k}{2},0).
	\end{align*}
\end{proposition}

The fact that $\ol\rho_f$ is reducible is equivalent to the fact that $f$ has (partial) Eisenstein descent described in~\cite[section 3.6]{Kri16}, meaning there is a congruence\[
\theta^j f\equiv \theta^j G\ \pmod{ \frp},\ j\geq 1
\]
where $\theta$ is the Atkin-Serre operator described in section 3.2 of \textit{op.\ cit.} and $G$ is a certain Eisenstein series indexed by $\phi$ and $\psi$. By the arguments in~\cite[Theorem 2.2.1]{CGLS}, the above congruence in turn yields the congruence between the $p$-adic $L$-functions\[
\cL_f^\BDP\equiv(\mathscr{E}^{\iota}_{\phi,\psi})^2\cdot(\cL_\phi)^2 \pmod{p\Lambda_K^{ur}},
\]
where $\mathscr{E}^{\iota}_{\phi,\psi}$ corresponds to the $\mathcal{P}_w(\theta)$ factors appearing in~\cref{algmain}. One knows that $\lambda(\cL_\phi)=\lambda(\cL_\psi)$ from a functional equation and $\mu(\cL_\phi)=\mu(\cL_\psi)=0$ by a result of Hida (\cite{Hida2010}).

Further computation as in~\cite[Theorem 2.2.2]{CGLS} gives the following comparison of analytic $\lambda$-invariants.

\begin{theorem}\label{anamain}
	Assume that $\ol\rho_f^{\mathrm{ss}}=\F(\phi)\oplus\F(\psi)$ as $G_\Q$-modules, with the characters $\phi$, $\psi$ labeled so that $p\nmid cond(\phi)$. Then $\mu(\cL_f^\BDP)=0$ and\[
	\lambda(\cL_f^\BDP)=\lambda(\cL_\phi)+\lambda(\cL_\psi)+\sum_{w\in S}\{\lambda(\cP_w(\phi))+\lambda(\cP_w(\psi))-\lambda(\cP_w(f))\}.\]
\end{theorem}

Combining~\cref{anamain} with~\cref{algmain}, together with the Iwasawa Main Conjectures for the characters proved by Rubin in~\cite{Rubin1991} (see~\cite[Theorem 2.2.3]{KY24} for a discussion about the trivial character, where we have $\lambda(\fX_\mathbf{1})=\lambda(\cL_\mathbf{1})+1$ instead), one knows for $\theta\ne\mathbf{1}$, \[\lambda(\fX_\theta)=\lambda(\cL_\theta)\]
and\[
\mu(\fX_\theta)=\mu(\cL_\theta).
\] We arrive at the first ingredient into the proof of the Iwasawa Main Conjectures.

\begin{theorem}\label{invariants}
		Assume that $\ol\rho_f^{ss}=\F(\psi)\oplus\F(\phi)$. Then $\mu(\cL_f^\BDP)=\mu(\frX_f)=0$ and\[
		\lambda(\cL_f^\BDP)=\lambda(\frX_f).\]
	\end{theorem}

\subsection{The anticyclotomic Iwasawa Main Conjectures}

Recall that $K$ is a field satisfying the Heegner hypothesis~\eqref{heeg}. Further assume that $D_K\ne -3$ is odd. The main results in this section come from~\cite[Section 3]{KY24}, which are built upon earlier works in~\cite{CGLS} and~\cite{CGS}.

To prove the Iwasawa Main Conjecture~\eqref{imcintro} in the introduction, we first notice that it is equivalent to a `Heegner Point Main Conjecture' of Perrin-Riou type. In fact, we have the following (the notations come from~\cite{CGLS}):
\begin{proposition}\label[proposition]{imcequiv}
	Assume that $p=v\ol v$ splits in $K$ and $\H^0(K,\ol\rho_f)=0$. Then the following are equivalent:\begin{enumerate}
		\item[(IMC1)] Both $\H^1_{\cF_{\Lambda_K}}(K,\bT)$ and $\cX=\H^1_{\cF_{\Lambda_K}}(K,M_f)^\vee$ have $\Lambda$-rank one, and the equality \begin{equation*}
			\Char_{\Lambda_K}(\cX_{\tors})\supset \Char_{\Lambda_K}(\H^1_{\cF_{\Lambda_K}}(K,\bT)/\Lambda_K\cdot\kappa_{\infty})^2
		\end{equation*}
		holds in $\Lambda_K$.
		\item[(IMC2)] Both $\H^1_{\cF_{\nr}}(K,\bT)$ and $\fX_f=\H^1_{\cF_{\nr}}(K,M_f)^\vee$ are $\Lambda_K$-torsion, and the equality\begin{equation*}
			\Char_{\Lambda_K}(\fX_f)\Lambda_K^{\nr}\supset(\cL_f^\BDP)
		\end{equation*}
		holds in $\Lambda_K^{\nr}$.
	\end{enumerate}
Moreover, the same result holds for opposite divisibilities.
\end{proposition}
\begin{proof}
	This is~\cite[Proposition 4.2.1]{CGLS}.
\end{proof}

Using an analog of Howard's Kolyvagin system argument in~\cite{How04}, the above divisibility in (IMC1) was mostly proved in~\cite{CGLS} (where they need to invert the height one prime $(\gamma-1)\subset \Lambda_K$ in the sense that the divisibility only holds in $\Lambda_K[1/(\gamma-1)]$. Their theorems only stated the results in $\Lambda_K[1/p,1/(\gamma-1)]$ since inverting $p$ was enough for their application, but the argument is already known to work at $p$ in~\cite{How04}) and later completed in~\cite{CGS}. Further modification of the Kolyvagin system argument for modular forms of higher weight $k=2r$ where odd $r$ is discussed in~\cite{KY24}.

From the above equivalence, the divisibility in (IMC2) is obtained. But from~\cref{invariants}, the divisibility in (IMC2) must in fact be an equality, and hence the same result holds for (IMC1). Finally, as is discussed in~\cite[Remark 3.0.9]{KY24}, the assumption $\H^0(K,\ol\rho_f)$ can be removed by Ribet's lemma, and we thus have\begin{theorem}\label[theorem]{IMC}
	Assume $f$ has weight $2r$ with $r$ odd. Assume that $p=v\ol v$ splits in $K$. Then the following statements hold:\begin{enumerate}
		\item[(IMC1)] Both $\H^1_{\cF_{\Lambda_K}}(K,\bT)$ and $\cX=\H^1_{\cF_{\Lambda_K}}(K,M_f)^\vee$ have $\Lambda$-rank one, and the equality \begin{equation*}
			\Char_{\Lambda_K}(\cX_{\tors})= \Char_{\Lambda_K}(\H^1_{\cF_{\Lambda_K}}(K,\bT)/\Lambda_K\cdot\kappa_{\infty})^2
		\end{equation*}
		holds in $\Lambda_K$.
		\item[(IMC2)] Both $\H^1_{\cF_{\nr}}(K,\bT)$ and $\fX_f=\H^1_{\cF_{\nr}}(K,M_f)^\vee$ are $\Lambda_K$-torsion, and the equality\begin{equation*}
			\Char_{\Lambda_K}(\fX_f)\Lambda_K^{\nr}=(\cL_f^\BDP)
		\end{equation*}
		holds in $\Lambda_K^{\nr}$.
	\end{enumerate}
\end{theorem}

Some consequences of the anticyclotomic Iwasawa Main Conjectures include the $p$-converse theorem ((IMC1)) and some partial results towards $p$-part BSD formulae ((IMC2)). In fact, (IMC2) can be used to show that a rank $0$ $p$-part BSD formula implies a rank $1$ formula (more precisely, we need a rank $0$ formula for $f^K$, the twist of $f$ by $K$, see~\cite[Theorem 5.3.1]{CGLS}). However, to get a rank $0$ result, we need a cyclotomic Iwasawa Main Conjectures over $\bQ$.

\subsection{Cyclotomic Iwasawa theory and main conjectures}\label{cycIwa}

	In the groundbreaking work~\cite{GV00}, Greenberg and Vastal studied the cyclotomic Iwasawa main conjectures for elliptic curves in the residually reducible case for the first time. Let $E$ be an elliptic curve and $p$ be an Eisenstein prime for $E$, then again there is an exact sequence
\[0\to \F_p(\phi) \to E[p] \to \F_p(\psi)  \to 0.\]
Under the assumption that
\begin{equation*}\tag{GV}\label{GV}\phi\text{ is either unramified at }p\text{ and odd, or ramified at }p\text{ and even,}\end{equation*}they proved the main conjecture\[
\Char_{\Lambda_\bQ}(\fX_{ord}(E/\bQ_\infty))=(\cL_p^{\MSD}(E/\bQ))\text{ in }\Lambda\]
where $\fX_\ord(E/\bQ_\infty)=\Sel_{p^\infty}(E/\bQ_\infty)^\vee$ is the dual of the $p$-primary Selmer group of $E$ and $\cL_p^{\MSD}$ is the Mazur--Swinnerton-Dyer $p$-adic $L$-function. Here $\Lambda\defeq\Lambda_\bQ$ is the cyclotomic Iwasawa algebra. In particular, their extra assumption guarantees that the $\mu$-invariants of both side are vanishing.

Their arguments can be generalize to higher weight modular forms, as we will explain later. However, since we do not expect the $\mu$ invariants to be always vanishing, new ideas are needed to cover the remaining cases. The first attempt along this line was in~\cite{CGS}, where they relaxed the assumption~\eqref{GV}, and they were able to compare the $\mu$ invariants of the algebraic side and the analytic side without assuming their vanishing. The results in~\cite{CGS} were obtained for elliptic curves only, under a weaker assumption that when restricted to the decomposition group $G_p$, $\phi|_{G_p}\ne \mathbf{1},\omega$, which is a weaker assumption (i.e., it is implied by~\eqref{GV}). The removal of the assumptions on characters is done in~\cite{KY24}. We now know the following

\begin{theorem}[cyclotomic IMC]\label{CGSMC}
	Let $f\in S_{2}^{new}(\Gamma_0(N))$ be a newform. Let $p\nmid 2N$ be an Eisenstein prime for $f$, i.e., $\ol\rho_f$ is reducible. Then 
	\[\Char_{\Lambda_\bQ}(\H^1_\Gr(\bQ,M_f')^\vee)=(\cL_f^{\MTT})\text{ in }\Lambda.\]
\end{theorem}
Here $\cL_f^\MTT$ is the Mazur--Tate--Teitelbaum $p$-adic $L$-function for modular forms and $\H^1_\Gr(\bQ,M_f')$ is analogous to the Greenberg's Selmer group defined in~\cref{section:control} where $V_f$ is replaced by $\rho_f$. In other words, the Selmer group here is for the representation attached to $f$ before the self-dual twist when $f$ is a newform of weight $>2$. When $f$ is an elliptic modular form of weight $2$, this Selmer group agrees with the one at the beginning of this section.

One would hope to have a higher weight analog of this theorem, based on generalizations of~\cite{CGS}. However, currently this result is not known. We therefore will stick to a weaker version where we still impose the conditions on the characters. Due to some technical difficulty, we further assume that the weight is at most $p-1$. In such situations, a relevant result is given in~\cite[Theorem 0.2]{Hir18}.

\begin{theorem}[cyclotomic IMC without self-dual twist]\label{cycMC}
	Let $f\in S_{2r}^{new}(\Gamma_0(N))$ be a newform. Let $p\nmid 2N$ be an Eisenstein prime for $f$, i.e., such that $\ol\rho_f$ is reducible. Assume that $2\leq 2r\leq p-1$. Further assume that the sub-representation of $\ol\rho_f$ is ramified and even and the quotient-representation is unramified and odd. Then 
	\[\Char_{\Lambda_\bQ}(\H^1_\Gr(\bQ,M_f')^\vee)=(\cL_f^{\MTT})\text{ in }\Lambda\otimes\bQ_p .\]Here $\cL_f^\MTT$ is the usual Mazur--Tate--Teitelbaum $p$-adic $L$-function in~\cite[Section 3.3]{Hir18}.
\end{theorem}

\begin{proof}
	This is the main result of~\cite{Hir18}, where the canonical periods for higher weight modular forms at Eisenstein primes are studied. It should be noted that in the proof of this theorem in~\textit{loc. cit.}, only the Iwasawa $\lambda$-invariants were compared and hence only an equality up to $p$-powers was obtained. The ambiguity in $\mu$-invariants comes from the failure of the comparison of the periods $\Omega^+_f$ and $\Omega^+_g$ in Theorem 3.19 in \textit{loc. cit.}. 
\end{proof}

Because of the ambiguity in $p$-powers, we fail to obtain the full cyclotomic Main Conjecture. In fact, as we will see, the lack of knowledge about periods is essentially the only issue preventing us from proving rank $0$ TNC formulae. However, one should also notice that this direct generalization of~\cite{GV00} is only good for the value $L(f,1)$, mainly due to a `branch' issue. Given that we would like to say something about $L'(f,r)$, we first need to study a different cyclotomic Main Conjecture for the right branch, as we now explain.

Kato~\cite{Kato} in fact considered the equivariant cyclotomic Iwasawa Main Conjectures for higher weight modular forms, which is a generalization of~\cref{CGSMC} already in weight $2$. More precisely, Kato studied his Selmer groups and $p$-adic $L$-functions over the bigger Iwasawa algebra $\widehat{\Lambda}\defeq\widehat{\Lambda_\bQ}\defeq \Gal(\bQ(\mu_{p^\infty})/\bQ)$, which decomposes into $p-1$ copies of the usual Iwasawa algebra $\widehat{\Lambda}=\bigoplus_{i=0}^{p-2}\Lambda^{(i)}$ where $\Delta=\Gal(\bQ(\mu_p)/\bQ)$ acts on $\Lambda^{(i)}$ via the $i$-th power of the the Teichmüller character $\omega$, and each module $M$ over $\widehat{\Lambda}$ also admits such a decomposition $M=\bigoplus_{i=0}^{p-2} M^{(i)}$. Consequently, his Main Conjecture is equivalent to $p-1$ Main Conjectures for different branches (see e.g. \cite[Section 5.1]{EPW06}), and he proved one divisibility (see Theorem 5.1.2 in \textit{loc. cit.}). The $i=0$ branch of the Main Conjectures corresponds to~\cref{CGSMC} in weight $2$, and has also been widely studied for modular forms of weight $k$ in~\cite{SU14} in the irreducible case (with $k\equiv 2 \pmod{p-1}$) and in~\cite{Hir18} (with $2\leq k\leq p-1$) in the Eisenstein case. Note that, however, despite the fact that each branch of the $p$-adic $L$-function does interpolate all values $L(f,1)$($=L(\rho_f,1)$, same below) to $L(f,k-1)$, if one wants to study the Tagamawa Number Conjecture concerning the value $L(f,s)$, the branch must be chosen carefully (in fact, one needs the $i$-th branch with $i\equiv 1-\frac{k}{2}\pmod{p-1}$, which will be clear in the next subsection). For example, in weight $2$ cases, the $i=0$ branch of the Main Conjectures is used in the proof of the BSD conjectures concerning $L(f,1)$.

The study of rank $1$ BSD formula is also possible due to the Gross--Zagier formula, which relates $L'(f,1)$ to heights of Heegner points in weight $2$. For higher weight modular forms, however, only the value $L'(f,\frac{k}{2})$ appears in Zhang's formula. Thus to study rank $1$ Tamagawa Number Conjectures, it is only reasonable to expect a formula involving $L'(f,\frac{k}{2})$, and we must prove the $(p-\frac{k}{2})$-th branch of the Main Conjectures, which was not explicitly written down in~\cite{Hir18}. Luckily, Hirano's arguments can be modified to cover these new cases, which we will explain in the next subsection.

\subsection{The $(p-\frac{k}{2})$-th branch and the self-dual twist}	
Since our aim is to study both rank $0$ and $1$ results towards the Tamagawa Number Conjecture, we will focus on the central $L$-values $L(f,r)$ and $L'(f,r)$ of a newform $f$ of weight $k=2r$. Therefore we will study the cyclotomic Iwasawa Main Conjectures for the $(p-r)$-th branch, and we will first consider the algebraic side and the analytic side separately. The study of Selmer groups is similar to those in~\cite[Section 4]{EPW06} and~\cite[Section 3.2]{Hir18}, and the study of $p$-adic $L$-functions will be a modification of~\cite[Section 3.3, Section 3.4]{Hir18}.

Note that as in the weight $2$ case, a proof of the rank $1$ TNC formula will crucially rely on the anticyclotomic Main Conjectures for $f$, for which it is necessary to consider the self-dual representation $\rho_f(1-r)$. For this reason, we will also need to study the cyclotomic Main Conjectures for self-dual representations. One should notice that both~\cite{EPW06} and~\cite{Hir18} work with the representation $\rho_f$ without self-dual twist (which thus has an unramified quotient), but thanks to the following lemma (here $\Lambda=\Lambda_\bQ$ is taken as the cyclotomic Iwasawa algebra), we don't have to redo all their work for $\rho_f(1-r)$.

\begin{lemma}\label[lemma]{Rubintwist}
	Let $A,B$ be finitely generated $\Lambda$-modules, and $\rho: \Gamma_\bQ\to \cO^\times$ a character. If $\Char_{\Lambda_\bQ}(A)=\Char_{\Lambda_\bQ}(B)$, then $\Char_{\Lambda_\bQ}(A\otimes \rho)=\Char_{\Lambda_\bQ}(B\otimes \rho)$.
\end{lemma}
\begin{proof}
	This follows immediately from~\cite[Lemma VI.1.2]{RubinEulerSystems}.
\end{proof}
For our application, we will choose, for the representation $\rho_f\otimes \omega^{p-r}$, $A=\Sel(\bQ_\infty,A_{\rho_f\otimes \omega^{p-r}})$ and $B=(\cL_{\rho_f\otimes \omega^{p-r}})$ (to be defined in the next two subsubsections) and $\rho=\langle\chi_\cyc\rangle^{p-r}$ where $\langle\chi_\cyc\rangle$ is the composition $\Gamma\to \Gal(\bQ(\mu_{p^\infty})/\bQ)\xrightarrow{\chi_\cyc} \cO^\times$, so that if we can prove the Main Conjecture for $f\otimes \omega^{p-r}$ (i.e., for the $(p-r)$-th branch of the Main Conjecture over $\widehat{\Lambda}$):\[
\Char_\Lambda(\Sel(\bQ_\infty,A_{\rho_f\otimes \omega^{p-r}})^\vee)=(\cL_{\rho_f\otimes \omega^{p-r}})\text{ in }\Lambda,
\]
then the Main Conjecture for $\rho_f\otimes \omega^{p-r}\otimes \langle\chi_\cyc\rangle^{1-r}$ (here $\langle \chi_\cyc\rangle$ is also considered as a character of $\Gal(\bQ(\mu_{p^\infty})\bQ)$ via the projection $\Gal(\bQ(\mu_{p^\infty})/\bQ)\to \Gamma)$ also holds. But $\rho_f\otimes \omega^{p-r}\otimes \langle\chi_\cyc\rangle^{1-r}=\rho_f\otimes \omega^{1-r}\otimes \langle\chi_\cyc\rangle^{1-r}$ is nothing but the self-dual representation $\rho_f(1-r)$, so it suffices to prove the Main Conjecture for $\rho_f$ with the $(p-r)$-th branch.

As in~\cite{Hir18}, we assume the following assumption throughout this chapter, where~\eqref{GV'} is an equivalent description of~\eqref{GV} after fixing an order of the characters:
\begin{assumption}
	In the remaining of this chapter, we assume that
	\begin{itemize}
		\item $r$ is odd.	
		\item In the short exact sequence
		\begin{equation}\label{twistrhofbar}0\to \bF(\phi) \to \ol\rho_f\otimes\omega^{p-r} \to \bF(\psi)\to 0,\end{equation}
		\begin{equation*}\tag{GV'}\label{GV'}
			\phi=\omega^{p-r}\phi'\text{ where }\phi'\text{ is ramified at }p\text{ and even}.
		\end{equation*}
		Note that this means $\psi=\omega^{p-r}\psi'$ where $\psi'$ is unramified at $p$ and odd. Also we have $\widehat{\phi'}=\widehat{\psi'}^{-1}\chi_\cyc^{2r-1}$ where $\widehat{\phi'}$ and $\widehat{\psi'}$ are Teichmüller lifts of $\phi'$ and $\psi'$ respectively. Since $r$ is odd and $p-r$ is even, $\phi$ will be even and ramified, and $\psi$ will be odd and ramified.
	\end{itemize}
\end{assumption}

In particular, neither $\phi$ nor $\psi$ is unramified, which is different from the weight $2$ case.

\subsection{The Selmer groups}\label{p-rSel}
In this subsubsection, we adopt the conventions from~\cite[Section 4]{EPW06} so that our $A_{\rho_f\otimes \omega^{p-r}}$ is their $A_{f,p-r}$. The goal of this subsubsection is to describe the algebraic $\mu$- and $\lambda$- invariants of $\Sel(\bQ_\infty,A_{\rho_f\otimes \omega^{p-r}})$ which is the $\Sel(\bQ_\infty,A_{f,p-r})$ in \textit{loc. cit.} in the Eisenstein case. 

We begin by notice that the proof of Theorem 4.1.1 in \textit{loc. cit.} still applies in the Eisenstein case under~\eqref{GV'}. Notice that almost divisibility (i.e., that $\Sel(\bQ_\infty,A_{f,p-r})$ has no proper $\Lambda$-submodules of finite index) now follows from the fact that $\H^0(\bQ,\ol\rho_f\otimes \omega^{p-r})=0$ and hence $\H^0(\bQ,A_{f,p-r})=0$, which is in turn a consequence of~\eqref{GV'}. In particular, we know that $\mu^\alg(f,\omega^{p-r})=0$ if and only if $\Sel(\bQ_\infty,A_{f,p-r})[\pi]$ is finite, in which case $\lambda^\alg(f,\omega^{p-r})=\dim_{\F}(\Sel(\bQ_\infty,A_{f,p-r})[\pi])$. Here $\pi$ is a uniformizer of the ring of integers in a (possible extension of) the Fourier field of $f$.

There are two short exact sequences involving $A_{f,p-r}$:\[
0\to \ol\rho_f\otimes\omega^{p-r}\to A_{f,p-r} \xrightarrow{\pi} A_{f,p-r}\to 0
\]
\[
0\to \Fil^+(A_{f,p-r}) \to A_{f,p-r} \to \Fil^-(A_{f,p-r}) \to 0
\]
where $G_\bQ$ acts on $\Fil^+(A_{f,p-r})$ via $\chi_\cyc^{k-1}\lambda(\alpha)^{-1}\omega^{p-r}$ and on $\Fil^-(A_{f,p-r})$ via $\lambda(\alpha)\omega^{p-r}$, with $\lambda(\alpha)$ being unramified characters that sends the arithemtic Frobenius to the unit root $\alpha$ of the Hecke polynomial of $f$. One thing to notice is that again neither $\Fil^+(A_{f,p-r})$ nor $\Fil^-(A_{f,p-r})$ is unramified, which is different from the weight $2$ case.

To keep things consistent, we also slightly modify the definitions of the Selmer groups for the characters $\phi$ and $\psi$ in the equation~\eqref{twistrhofbar}. With notations from~\cite[Section 3.2]{Hir18}, we define $F^+(A_\phi)=A_\phi$ and $F^+(A_\psi)=0$ according to ramification of $\phi'$ and $\psi'$ (rather than $\phi$ and $\psi$). Then we can define the Selmer groups for $A_\phi$ and $A_\psi$ in the same manner, and the arguments from \textit{loc. cit.} can be used to show the equalities (their equations (3.3) and (3.4)):\[
\mu_{f\otimes\omega^{p-r}}^\alg=\mu_{f\otimes\omega^{p-r},\Sigma_0}^\alg=0
\]
and\begin{align*}
	\lambda^\alg_{f\otimes\omega^{p-r},\Sigma_0}&=\dim_{\cO/\pi}\Sel^{\Sigma_0}_{A_{f,p-r}[\pi]}(\bQ_\infty)\\
	&=\dim_{\cO/\pi}\Sel^{\Sigma_0}_{A_\phi[\pi]}(\bQ_\infty)+\dim_{\cO/\pi}\Sel^{\Sigma_0}_{A_\psi[\pi]}(\bQ_\infty)\\&=\lambda_{\phi,\Sigma_0}+\lambda_{\psi,\Sigma_0}.
\end{align*}
Note that in their proofs they only used the assumptions that $\phi$ is even and non-trivial, and $\psi$ is odd and not $\omega$, which are still true under~\eqref{GV'}, except~\cite[Theorem 3.10]{Hir18}, which is replaced by~\cite[Theorem 4.1.1]{EPW06} above.

\subsection{The $p$-adic $L$-functions}	
In this subsubsection, we will follow the discussions on $p$-adic $L$-functions in~\cite[Section 3.3]{Hir18}, and modify the arguments when needed. In particular, we will take his $\chi$ to be $\omega^{p-r}$ so that the $p$-adic $L$-function $\cL_{\rho_f\otimes \omega^{p-r}}\defeq \sL_p(f,\omega^{p-r},T)$ is the $(p-r)$-th branch of Kato's big $p$-adic $L$-function over $\widehat{\Lambda}$, which will be related to $\Sel(\bQ_\infty,A_{f,p-r})$ in the $(p-r)$-th branch of Kato's Main Conjectures.

We similarly consider the imprimitive $p$-adic $L$-function $\sL_p^{\Sigma_0}(f,\omega^{p-r},T)$, and there is a relation 
\[\sL_p^{\Sigma_0}(f,\omega^{p-r},T)=\sL_p(f,\omega^{p-r},T)\prod_{\ell\in\Sigma_0}P_\ell(T),\] with $P_\ell(T)$ defined by Proposition 3.9 in \textit{op. cit.} for $A=A_{f,p-r}$. 

The analytic Iwasawa invariants of $f\otimes \omega^{p-r}$, $\mu^\an_{f\otimes\omega^{p-r}}$, $\lambda^\an_{f\otimes\omega^{p-r}}$, $\mu^\an_{f\otimes\omega^{p-r},\Sigma_0}$, and $\lambda^\an_{f\otimes\omega^{p-r},\Sigma_0}$ are now defined in the usual way via the above $p$-adic $L$-functions.

\subsection{An integral divisibility}
It is shown in~\cite{GV00} that the $p$-adic $L$-function attached to an elliptic curve is integral (i.e., belongs to $\Lambda_\bQ$). In fact, they showed that its $\mu$ invariant is $0$. However, for higher weight modular forms, such an integrality result is not as clear. For completeness, we include a possibly new proof of this integrality result for modular forms of any weight $2r$ that might be helpful for future purposes. We begin by recalling some backgrounds.

Recall that Kato in~\cite{Kato} studied one divisibility of the cyclotomic Iwasawa Main Conjecture over the Iwasawa algebra $\widehat{\Lambda_\bQ}\defeq\bZ_p\llbracket\Gal(\bQ(\mu_{p^\infty})/\bQ)\rrbracket$ for the $\bZ_p^\times$-extension. In particular, he showed that certain Selmer group `divides' a $p$-adic $L$-function $\widehat{\cL_f}$ in $\widehat{\Lambda_\bQ}\otimes\bQ$ and when $\ol\rho_f$ is irreducible, $\widehat{\cL_f}$ is integral, i.e. $\widehat{\cL_f}\in\widehat{\Lambda_\bQ}$, and the divisibility also holds in $\widehat{\Lambda_\bQ}$. In particular, our $p$-adic $L$-function $\cL_f^{(p-r)}\defeq \cL_{\rho_f\otimes\omega^{p-r}}$ is the $(p-r)$-th branch of $\widehat{\cL_f}$.

As in~\cite{GV00}, we would also need to know $\cL_f^{(p-r)}$ is integral in the Eisenstein case. They showed it via explicit computation in the case of elliptic curves. The integrality result is later proved by Wuthrich in~\cite{Wut17} using a different method. In fact, Wuthrich first proved the integrality of Kato's zeta elements for elliptic curves in the Eisenstein case, and then deduce the integrality of the $p$-adic $L$-function $\widehat{\cL_f}$ as a consequence of Kato's integral divisibility. We will generalize Wuthrich's arguments to show the integrality of $\cL_f^{(p-r)}$ for a higher weight modular form $f$. Note that our result is not a full generalization since we only focus on $\Lambda_\bQ$ rather than the full $\widehat{\Lambda_\bQ}$, but it is sufficient for our purpose.

\begin{proposition}\label{intdiv}
	Assume~\eqref{GV'}. Then
	\[\Char_{\Lambda_\bQ}(\H^1_\Gr(\bQ,M_f'))\supset(\cL_f^{(p-r)})\text{ in }\Lambda_\bQ.\]
	In particular, $\cL_f^{(p-r)}\in\Lambda_\bQ$.
\end{proposition}
\begin{proof}
	This result would follow from the same arguments as in~\cite[Theorem 16]{Wut17} if we consider the $(p-r)$-th component of all the Iwasawa modules over $\widehat{\Lambda_\bQ}$, provided that we can show that (analogs of conditions in \textit{op. cit.}):
	\begin{enumerate}
		\item There is a Galois stable $\cO$-lattice $T(f)$ of $V(f)$ for which $\H^1(\bQ^\Sigma/\bQ, T(f)\otimes\Lambda_\bQ)$ is $\Lambda_\bQ$-free;
		\item The local integrality holds for this lattice, i.e., in the notation of \textit{op. cit.}, $(Z(T)_{(p-r),\frP})\subset \bH^1(T)_{(p-r),\frP}$ for all height $1$ prime $\frP$ of $\Lambda_\bQ$ (see~\cite[Lemma 12]{Wut17}).
	\end{enumerate} 
	where $V(f)$ is the $p$-adic representation attached to $f$ in~\cite{Wut17} (also in~\cite{Kato}), which is dual to our $\rho_f$. In particular, a choice of a Galois stable lattice $T(f)$ in $V(f)$ corresponds to a choice of $T_f$ in $V_f=\rho_f(1-r)$ via the relation $T(f)^\vee(1-r)=T_f$. In fact, we will show that any lattice $T_f$ satisfying~\eqref{GV'} would work. 
	
	Now let $T_f$ be any lattice such that satisfies~\eqref{GV'} (so for the residual representation of $T(f)$, neither the sub-representation nor the quotient-representation is trivial) and let $\H^1(\bT)\defeq\H^1(\bQ^\Sigma/\bQ, T(f)\otimes\Lambda_\bQ)$. We claim that $\H^1(\bT)[T]=0$ and $\H^1(\bT)/T$ is $\cO$-free, from which the result would follow from Nakayama's Lemma (see e.g.~\cite[Lemma 1.1.2]{KY24}). From the long coholomogy sequence applied to the short exact sequence\[
	0\to\bT\xrightarrow{\cdot T}\bT\to \bT/T(=T(f)) \to 0,
	\]
	we know there is a surjection $\H^0(\bQ^\Sigma/\bQ, T(f))\twoheadrightarrow \H^1(\bT)[T]$ and that $(\H^1(\bT)/T)_\tors$ is contained in $\H^1(\bQ^\Sigma/\bQ,T(f))_\tors$. However, $\H^0(\bQ^\Sigma/\bQ,T(f))=0$ since it's $\cO$-torsion free and $\H^0(\bQ^\Sigma/\bQ,T(f))/\frp\subset\H^0(\bQ^\Sigma/\bQ,T(f)/\frp)=\H^0(\bQ^\Sigma/\bQ,\ol{\rho(f)})=0$. The last equality follows from the fact that $\phi'|_{G_p},\psi'|_{G_p}\ne\mathbf{1}$. $\H^1(\bQ^\Sigma/\bQ,T(f))_\tors$ is also trivial since there is a surjection $0=\H^0(\bQ^\Sigma/\bQ, \ol{\rho(f)})\twoheadrightarrow\H^1(\bQ^\Sigma/\bQ,T(f))[\frp]$. Thus $\H^1(\bT)[T]=0$ and $\H^1(\bT)/T$ is $\cO$-torsion free (hence $\cO$-free). Hence (i) is satisfied.
	
	In fact, one can show (ii) holds for any lattice. By~\cite[Theorem 8]{ver23} (see also the end of section 2 in \textit{loc. cit.}), the local integrality will hold for any lattice over $\widehat{\Lambda_\bQ}$ if it holds for one. It is shown in~\cite[Theorem 13.14]{Kato} for Kato's `canonical lattice' $T_{\cO_\lambda}(f)$, so it indeed holds for any lattice. Now the same holds for $\Lambda_\bQ$ if we take the $(p-r)$-th component of the Iwasawa modules. Hence (ii) is also satisfied.

\end{proof}

\begin{remark}
	Another application of~\cref{Rubintwist} (with equality replaced by one divisibility) shows that the integral divisibility with the $(p-r)$-th branch $p$-adic $L$-function for $V(f)$ above is equivalent to an integral divisibility for the self-dual twist $V_f$.
\end{remark}

\subsection{The cyclotomic Main Conjectures}\label{cycMCp-r}
We now closely follow~\cite[Section 3.4]{Hir18} to study the cyclotomic Main Conjecture for $f\otimes\omega^{p-r}$. Namely, we want to show that\begin{equation*}\tag{(p-r)-th IMC}\label{p-rMC}
	\Char_\Lambda(\Sel(\bQ_\infty,A_{f,p-r})^\vee)=(\sL_p(f,\omega^{p-r},T))\text{ in }\Lambda.
\end{equation*}
The divisibility $\Char_\Lambda(\Sel(\bQ_\infty,A_{f,p-r})^\vee)\supset(\sL_p(f,\omega^{p-r},T))\text{ in }\Lambda\otimes\bQ_p$ is a consequence of Kato's equivariant divisibility specialized to the $(p-r)$-th branch. We will now compare the Iwasawa invariants of both sides to show that this must in fact be an equality. However, using existing arguments, we cannot show the equality in $\Lambda$.

The study of the analytic side is roughly done by relating the following $p$-adic $L$-functions:\[
\sL_f\sim\sL_g\sim\sL_G\sim\sL_\phi\sL_\psi, \] where $g$ is the eigenform defined by \[f\otimes \mathbf{1}_N=\sum_{(n,N)=1}a_n(f)q^n\] and $G$ is an Eisenstein series built from $\phi$ and $\psi$.

We first consider the $p$-adic $L$-functions for the two characters $\phi$ and $\psi$, similar to~\cite[equation (26)]{GV00}. The $p$-adic $L$-function $\sL_p(A_\phi,T)=\sL(\widehat{\phi'},\omega^{p-r},T)$ for $\phi$ is characterized by the interpolation property\[
\sL(\widehat{\phi'},\omega^{p-r},\zeta-1)=L(\widehat{\phi'}\omega^{p-r}\rho,1)=L(\widehat{\psi'}^{-1}\omega^{p-r}\rho,2-k)
\]for every non-trivial character $\rho$ of $\Gamma$ of finite order, and $\zeta=\rho(\gamma)$ where $\gamma$ is a fixed topological generator of $\Gamma$. It is related to the Kubota--Leopoldt $p$-adic $L$-function $L_p(\chi_\cyc^{k-2}\widehat{\psi'}^{-1}\omega^{p-r+1},s)$ by \[L_p(\chi_\cyc^{k-2}\widehat{\psi'}^{-1}\omega^{p-r+1},s)=\sL(\widehat{\phi'},\omega^{p-r},\langle\chi_\cyc\rangle(\gamma)^{-s}-1)\] for all $s\in \bZ_p$. By the Ferrero--Washington theorem, $\sL_p(A_\phi,T)\not\in p\Lambda$, and the Mazur--Wiles theorem implies that the $\lambda$-invariant of $\sL_p(A_\phi,T)$ is equal to the $\corank_\cO(S_{A_\phi}(\bQ_\infty))$, which is denoted by $\lambda_\phi$. The imprimitive $p$-adic $L$-function for $\phi$ is defined by\[
\sL^{\Sigma_0}_p(A_\phi,T)=\sL_p(A_\phi,T)\prod_{\ell\in\Sigma_0}(1-\widehat{\psi'}^{-1}\omega^{p-r}(\ell)\ell^{k-2}(1+T)^{f_\ell}).
\]
Here $f_\ell\in\bZ_p$ is determined by $\gamma_\ell=\gamma^{f_\ell}$, where $\gamma_\ell$ is the Frobenius element corresponding to the prime $\ell\in\Gamma$.

The $p$-adic $L$-function $\sL_p(A_\psi,T)=\sL(\widehat{\psi'},\omega^{p-r},T)$ for $\psi$ is characterized by the interpolation property\[
\sL(\widehat{\psi'},\omega^{p-r},\zeta-1)=\tau(\omega^{r-1}\widehat{\psi'}^{-1}\rho^{-1})L(\widehat{\psi'}\omega^{p-r}\rho,1)/(2\pi i)=\frac{1}{2}L(\omega^{r-1}\widehat{\psi'}^{-1}\rho^{-1},0)
\]
for every non-trivial character $\rho$ of $\Gamma$ of finite order. Here $\tau(-)$ denotes the Gauss sum.

$\sL_p(A_\psi,T)$ is related to the Kubota--Leopoldt $p$-adic $L$-function $L_p(\widehat{\psi'}^{-1}\omega^{r},s)$ by\[
L_p(\widehat{\psi'}^{-1}\omega^{r},s)=\frac{1}{2}\sL(\widehat{\psi'},\omega^{p-r},\langle\chi_\cyc\rangle(\gamma)^{-s}-1)
\]
for any $s\in \bZ_p$. Again, by the Ferrero--Washington theorem, $\sL_p(A_\psi,T)\not\in p\Lambda$ and by the Mazur--Wiles theorem, its $\mu$ invariant is equal to $\corank_{\cO}(\Sel_{A_\psi}(\bQ_\infty))$. Furthermore, we define the imprimitive $p$-adic $L$-function by\[
\sL^{\Sigma_0}_p(A_\psi,T)=\sL_p(A_\psi,T)\prod_{\ell\in\Sigma_0}(1-\widehat{\psi'}\omega^{p-r}(\ell)\ell^{-1}(1+T)^{f_\ell}).
\]

We now define the $p$-adic $L$-function for the Eisenstein series $G$ associated to $\phi$ and $\psi$ in the sense of~\cite[Lemma 3.18]{Hir18}, similar to~\cite[equation (28)]{GV00}. More precisely, $G$ is determined by\[
L(G,s)=L_{\Sigma_0}(\widehat{\psi'},s)L_{\Sigma_0}(\widehat{\psi'}^{-1},s-k+1).
\]
The $p$-adic $L$-function $\sL_p(G,T)=\sL(G,\omega^{p-r},T)$ is characterized by the interpolation property\begin{align*}
	\sL(G,\omega^{p-r},\zeta-1)&=\tau(\omega^{r-1}\widehat{\psi'}^{-1}\rho^{-1})L(G,\omega^{p-r}\rho,1)/(2\pi i)\\&=(L_{\Sigma_0}(\widehat{\psi'}^{-1}\omega^{p-r}\rho,2-k))(\tau(\omega^{r-1}\widehat{\psi'}^{-1}\rho^{-1}) L_{\Sigma_0}(\widehat{\psi'}\omega^{p-r}\rho,1)/(2\pi i))
\end{align*}
for every non-trivial character $\rho$ of $\Gamma$ of finite order. It is then clear that \[\sL_p(G,T)=\sL^{\Sigma_0}_p(A_\phi,T)\sL^{\Sigma_0}_p(A_\psi,T).\]
Thus the $\mu$-invariant of $\sL_p(G,T)$ is zero since it's true for both factors and its $\lambda$-invariant is $\lambda_{A_\phi,\Sigma_0}+\lambda_{A_\psi,\Sigma_0}$ which agree with $\lambda^\alg_{f\otimes\omega^{p-r},\Sigma_0}$ from the subsubsection on Selmer groups.

It remains to relate $\sL(f,\omega^{p-r},T)$ to $\sL(G,\omega^{p-r},T)$ via $\sL(g,\omega^{p-r},T)$ as in~\cite[Theorem 3.19]{Hir18}. From now on, we assume that\begin{equation*}\tag{restricted weights}
	2\leq k\leq p-1.
\end{equation*} Until the end of this subsubsection, all references are to \textit{op. cit.}. First, we see that there is a relation\begin{equation}\label{LfLg}
	\sL_p(g,T)=\frac{\Omega^+_f}{\Omega^+_g}\sL(f,T).
\end{equation}
The congruence $f\equiv g\equiv G\pmod{\pi}$ still holds since our $f$, $g$ and $G$ are the same as those in \textit{op. cit.}, just with twisted $p$-adic $L$-functions. The proof remains the same before the application of Theorem 2.10, where now we apply the theorem to the triples $(g,G,\omega^{p-r}\rho)$ instead of $(g,G,\rho)$, and there is a unit $u'\in \cO^\times$ such that there is a congruence\[
\tau(\overline{\omega^{p-r}\rho})\frac{L(g,\omega^{p-r}\rho,1)}{(2\pi i)\Omega^+_g}\equiv u'\tau(\overline{\omega^{p-r}\rho})\frac{L(G,\omega^{p-r}\rho,1)}{2\pi i} \pmod{\pi}.
\] 
for every character $\rho$ of $\Gamma$ of finite order such that $\omega^{p-r}\rho$ is non-trivial and non-exceptional with conductor $p^{\nu_p}$. There clearly exist infinitely many such $\rho$ since we can simply consider characters of the form $\omega^{r-p}\rho_0$ with $\rho_0$ non-trivial and non-exceptional.

Now from the same computation in Theorem 3.19, we have\begin{align*}
	\sL(g,\omega^{p-r},\zeta-1)&=\tau(\omega^{r-1}\rho^{-1})\alpha(p,g)^{-\nu_p}\frac{L(g,\omega^{p-r}\rho,1)}{(2\pi i)\Omega^+_g}\\
	&\equiv u'\tau(\omega^{r-1}\rho^{-1})\alpha(p,g)^{-\nu_p}\frac{L(G,\omega^{p-r}\rho,1)}{(2\pi i)}\\
	&\equiv u'\tau(\overline{\widehat{\psi'}})^{-1}\overline{\omega^{p-r}\rho}(m_\psi)^{-1}\tau(\overline{\widehat{\psi'}}\overline{\omega^{p-r}}\overline{\rho})\frac{L(G,\omega^{p-r}\rho,1)}{(2\pi i)}\\
	&\equiv u'\tau(\overline{\widehat{\psi'}})^{-1}(1+T)^{-n_\psi}\sL(G,\omega^{p-r},\zeta-1)\pmod{\pi}
\end{align*}
for every $\zeta=\rho(\gamma)\ne 1$ since $\overline{\theta}=\theta^{-1}$ for any of the characters. 

Thus Theorem 3.19 also holds in our situation, and we get the equalities \[\lambda^\an_{f\otimes\omega^{p-r},\Sigma_0}=\lambda(\sL(G,\omega^{p-r},T))=\lambda_{A_\phi,\Sigma_0}+\lambda_{A_\psi,\Sigma_0}=\lambda^\alg_{f\otimes\omega^{p-r},\Sigma_0}\] and \[\mu^\an_{\sL_p(g,T)}=\mu(\sL(G,\omega^{p-r},T))=0=\mu^\alg_{f\otimes\omega^{p-r},\Sigma_0}.\]
However, we fail to get information about $\mu^\an(\sL_p(f,T))$ because we can't determine the order of $p$ in $\frac{\Omega^+_f}{\Omega^+_g}$ in equation~\eqref{LfLg}. Note that in the case of elliptic curves, we know the quotient is a $p$-adic unit thanks to~\cite[Lemma 3.6]{GV00}.

Thus from Theorem 3.17, we also obtain the equalities in the primitive setting:\[
\lambda^\an_{f\otimes\omega^{p-r}}=\lambda^\alg_{f\otimes\omega^{p-r}}
\] 
which combined with Kato's divisibility finished the proof of the $(p-r)$-th branch of the cyclotomic Main Conjecture~\eqref{p-rMC}, up to $p$-powers. Thus we obtain the partial result towards the cyclotomic Main Conjecture for the self-dual representation $\rho_f(1-r)$ from~\cref{Rubintwist}. In summary, we have proved the following:
\begin{theorem}[cyclotomic IMC with self-dual twist]\label{cycMCselfdual}
	Let $f\in S_{2r}^{new}(\Gamma_0(N))$ be a newform. Let $p\nmid 2N$ be an Eisenstein prime for $f$, i.e., such that $\ol\rho_f$ is reducible. Assume that $2\leq 2r\leq p-1$ and that $r$ is odd. Further assume~\eqref{GV'}. Then 
	\[\Char_{\Lambda_\bQ}(\H^1_\Gr(\bQ,M_f)^\vee)=(\cL_f^\MTT(\langle\chi_\cyc\rangle^r\cdot-))\text{ in }\Lambda\otimes\bQ_p.\]Here $\H^1_\Gr(\bQ,M_f)$ is the Selmer group for the self-dual representation $V_f=\rho_f(1-r)$ to be defined in~\cref{section:control}, and $\cL_f^\MTT(\langle\chi_\cyc\rangle^r\cdot-)$ is the Mazur--Tate--Teitelbaum $p$-adic $L$-function corresponding to the self-dual twist $\rho_f(1-r)$.
\end{theorem}

We now discuss some potential applications of the above Main Conjecture. If we let $\cF\in\Lambda_\bQ$ be a generating power series of $\H^1_\Gr(\bQ,M_f)^\vee$, then~\cref{cycMCselfdual} says that $\cF\sim_p\cL\defeq\cL_f^\MTT(\langle\chi_\cyc^r\rangle\cdot-)$ in $\Lambda\otimes\bQ_p$. If we can prove $\mu(\cL)=0$ (say, if we can show $\Omega^+_f\sim_p\Omega^+_g$), we will get an equality in $\Lambda$. With~\cref{intdiv} at hand, we in fact only need a slightly weaker assumption:\begin{equation*}\tag{per}\label{per}
	\ord_p(\frac{\Omega^+_f}{\Omega^+_g})\geq 0
\end{equation*}

On the other hand, from the interpolation property of the $p$-adic $L$-function $\cL_f^{\MTT,(p-r)}$ (see e.g.~\cite[Theorem 16.2(ii)]{Kato}), one knows\begin{equation}\label{cycpL}
	\cL(0)\defeq \cL_f^\MTT(\langle\chi_\cyc\rangle^r\cdot\mathbf{1})=\cL_f^\MTT(\langle\chi_\cyc\rangle^{r})=(1-\frac{p^{r-1}}{\alpha})^2\cdot\frac{(r-1)!(2\pi i)^{r-1}}{\Omega_f}\cdot L(f,r)
\end{equation}
where $\Omega_f=\Omega^+_f$ is the period since $r$ is odd.

In~\cref{section:control}, we will study $\cF(0)\defeq \cF(\mathbf{1})$ and show that it is consistent with what we get on the analytic side towards the Tamagawa Number Conjecture.

\section{Heegner cycles and $p$-adic Abel--Jacobi maps}\label{2}
In this section, we introduce the \textit{Heegner cycles}, which are certain cycles in the Chow group of the Kuta--Sato variety whose images under the $p$-adic Abel--Jacobi in the Bloch--Kato Selmer groups are expected to be non-torsion. They appear in the ($p$-adic) Gross--Zagier formulae and eventually allow us to control the Tate-Shafarevich group.

\subsection{Kuga--Sato varieties}
We start by considering the Kuga--Sato varieties in two similar settings. Let $\cE(N)\to X(N)$ be the universal generalized elliptic curve over the compact modular curve $X(N)$ of level $\Gamma(N)$. The Kuga--Sato variety $\tilde{\cE}^{2r-2}(N)$ is then defined as the canonical desingularization of the $(2r-2)$nd fiber product of $\cE(N)$ with itself over $X(N)$.

In~\cite[section 2.1]{BDP13}, a different universal generalized elliptic curve $\cE_1(N)\to X_1(N)$ was considered, where $X_1(N)$ is the compact modular curve of level $\Gamma_1(N)$, thus yielding a different Kuga--Sato variety $W_{2r-2}\defeq\tilde{\cE}^{2r-2}_1(N)$ constructed in the same way. Cycles constructed from different Kuga--Sato varieties have been considered by different people. We will compare these cycles near the end of this work. In fact,~\cite{BDP13} introduced a generalization of the Kuga--Sato varieties.

\subsection{Masoero's Heegner cycles over $\Gamma(N)$}\label{Macycle}
In this section, we introduce Heegner cycles over $\Gamma(N)$ constructed in~\cite[Section 4.1]{Mas} (see also~\cite[Section 5.2]{Thackeray2022}). Again let $K=\bQ(\sqrt{-D})$ be an imaginary quadratic field satisfying~\cref{assum}.

The $N$-isogeny $\bC/\cO_K\to\bC/\frN^{-1}$ induces a Heegner point $x_1\in X_0(N)$ which is rational over the Hilbert class field $K_1$ of $K$ by the theory of complex multiplication. A lift $x\in\pi^{-1}(x_1)$ of $x_1$ under the canonical projection $\pi: X(N)\to X_0(N)$ corresponds to an elliptic curve $E_x$ of full level $N$ and complex multiplication by $\cO_K$. Fix the unique square root $\sqrt{-D}$ with positive imaginary quadratic part. Let $\Gamma_{\sqrt{-D}}\in E_x\times E_x$ be the graph of $\sqrt{-D}$ and let $i_x:\ol\pi^{-1}_{2r-2}(x)=E_x^{2r-2}\hookrightarrow \tilde{\cE}^{2r-2}(N)$. 

We call \[\Delta_N\defeq\Pi_B\Pi_\eps (i_x)_*(\Gamma_{\sqrt{-D}}^{r-1})\in\Pi_B\Pi_\eps\CH^r(\tilde{\cE}^{2r-2}(N)/K_1\otimes\bZ_p)\] Masoero's cycle, where the projector $\Pi_B\Pi_\eps$ are as in~\cite[Section 2.1, Section 3.1]{Mas} (see also~\cite[Section 5.1]{Thackeray2022}).

\subsection{Zhang's cycles}\label{zhangcycle}
For later use, we also briefly explain the cycles constructed by Zhang in~\cite{Zhang1997} built from Heegner cycles from the previous section.  We follow the arguments in~\cite[Section 4.1 and Section 4.2]{LV23}. We denote $\tilde{\cE}^{2r-2}(N)$ by $W'_{2r-2}$.

Let $E_x$ be the elliptic curve as before and let $Z(x)$ be the divisor class on $E_x\times E_x$ of $\Gamma_{\sqrt{-D}}-E_x\times\{0\}\cup\{0\}\times E_x$. Let $\tilde{\Gamma}$ denote the cycle\[\Pi_B\Pi_\eps (i_x)_*(Z(x)^{r-1})\in\Pi_B\Pi_\eps\CH^r(W'_{2r-2}/K_1)
\]

Let $W_{2r}(E_x)$ denote the cycle\[
\sum_{g\in G_{2r-2}}\sgn\ g^*(Z(x)^{r-1})\in \CH^{r}(W'_{2r-2})_\bQ,
\]
where $G_{2r-2}$ denotes the symmetric group of $2r-2$ letters which acts on $E_x^{2r-2}$ by permuting the factors. 

Then from Lemma 4.1 in \textit{op.\ cit.}, we get the relation\begin{equation}\label{zhangcyc}
\tilde{\Gamma}=\frac{\Pi_B W_r(E_x)}{(2r-2)!} \in \Pi_B \Pi_\epsilon \CH^{r}(W'_{2r-2})_\bQ.\end{equation}

Zhang's cycle $S_{2r}(E_x)$ with real coefficients is defined by\[
S_{2r}(E_x)\coloneq c\cdot W_{2r}(E_x),
\]
where $c\in \bR$ is a positive constant such that the self-intersection of $S_{2r}(E_x)$ on each fiber is equal to $(-1)^{r-1}$. In fact, from~\cite[Equation (4.12)]{LV23}, we know\[
c=\frac{1}{(r-1)!\cdot\sqrt{(2r-2)!}\cdot (\sqrt{-2D_K})^{r-1}}.
\]

Zhang's cycles are closely related to Masoero's cycles, as we will see in~\cref{cyclecomp}.

\subsection{BDP's Heegner cycles over $\Gamma_1(N)$}\label{2.2}

In this subsection we study a special case of the \textit{generalized Heegner cycles} introduced in~\cite{BDP13}. We will follow the construction in~\cite[Section 4]{BDP2017}. Note that for our purpose, it is enough to consider the \textit{classical} Heegner cycles corresponding to $r_1=2r-2,r_2=0$ in the notations of \textit{loc.\ cit.}.

Recall that $f\in S^{new}_{2r}(\Gamma_0(N))$ is a newform of weight $2r\geq 2$. We continue to assume the Heegner hypothesis~\eqref{heeg}, which guarantees the existence of an ideal $\mathcal{N}\subset\cO_K$ with\[
\cO_K/\mathcal{N}\iso \bZ/N\bZ.\]
Let $A=\bC/\cO_K$ be an elliptic curve defined over the Hilbert class field $K_1$ of $K$ with CM by $K$ and fix a generator $t$ of the cyclic group $A[\cN]$. The pair $(A,t)$ then defines a point $P$ on the modular curve $X_1(N)$ which is defined over an abelian extension of $K$.

For an ideal $\fra\in\cO_K$, write $A_\fra$ for the elliptic curve $\bC/\fra^{-1}$ and let $\phi_\fra$ denote the canonical isogeny of degree $N\fra$,\[
\phi_\fra: A=\bC/\cO_K\to\bC/\fra^{-1}=A_\fra.
\]

Let $W_{2r-2}$ be the Kuta--Sato variety of dimension $2r-1$ over $X_1(N)$. We now construct a cycle\[
\Delta_\fra\in\CH^{r}(W_{2r-2}/K_1)_{0,\bQ}
\]
for every ideal $\fra\in \cO_K$ that is prime to $\cN$.

Let $t_\fra$ denote the image of $t$ under the map $\phi_\fra$. Then the pair $(A_\fra,t_\fra)$ defines a point $P_\fra$ on the modular curve $X_1(N)$. The fiber of $W_{2r-2}$ over $P_\fra$ is canonically isomorphic to $A_\fra^{2r-2}$.
Now define\[
\Gamma_\fra=(graph\ of\ \sqrt{-D})^{tr} \in Z^1(A_\fra\times A_\fra)
\]
and let \[
\Delta_\fra\defeq \epsilon_W(\Gamma_\fra^{r-1})\in\CH^{r}(W_{2r-2}/K_1)_\bQ.\]
Here $\epsilon_W$ is the projectors on $W$ described in~\cite[Section 2.1]{BDP13}. It should be noted that $\Delta_\fra$ can be shown to be homologically trivial on $W_{2r-2}$ using the arguments in Section 2.2 and section 2.3 in \textit{op.\ cit.}.

We mention that the above cycle is different from the one considered in~\cite{BDP13} in that they consider the cycles $\Delta_\fra^\BDP$ corresponding to $r_1=r_2=2r-2$ that live in $\CH^{4r-3}(X_{2r-2}/K_1)_{0,\bQ}$, where $X_{2r-2}=W_{2r-2}\times A^{2r-2}$. According to~\cite[Section 2.4]{BDP13}, $\Delta_\fra^\BDP$ contains at least as much information as $\Delta_\fra$. We will come back to the comparison of the cycles in~\cref{cyclecomp}.

Finally, we remark that the Heegner point $\kappa_\infty$ in~\cref{IMC}(IMC1) can be compared to certain Heegner class $\kappa_1$ (see~\cite[Remark 4.1.3]{CGLS}) which in turn is essentially constructed from the cycles $\Delta_\fra^\BDP$ (see for example, ~\cite[section 4]{CastellaHsieh}). In~\cref{cyclecomp} we will see that under some reasonable hypothesis, (Abel--Jacobi image of) $\Delta_\fra$ is non-torsion if and only if (that of) $\Delta_\fra^\BDP$ is non-torsion.

\subsection{The Bloch--Kato logarithm}\label{log}
To define the $p$-adic Abel--Jacobi maps, we first recall the Bloch--Kato logarithm studied in~\cite{BK07}. We first recall some definitions from $p$-adic Hodge theory.

Let $F$ be a finite extension of $\bQ_p$ and let $V$ be a finite dimensional $G_F$-representation. Let $\bB_\dR$ be Fontaine's ring of $p$-adic periods and let $\bD_\dR(V)\coloneq (V\otimes_{\bQ_p}\bB_\dR)^{G_F}$. Then $\bD_\dR(V)$ is a $F$-vector space equipped with a decreasing filtration\[
\{\Fil^r\bD_\dR(V)\}_{r\in\bZ}\]
satisfying $\cup\Fil^r\bD_\dR(V)=\bD_\dR(V)$ and $\cap\Fil^r\bD_\dR(V)=0$. We say that $V$ is a de Rham representation if $\dim_{\bQ_p}(\bD_\dR(V))=\dim_{\bQ_p}(V)$.

For a de Rham representation, The Bloch--Kato exponential map is a morphism\[
\exp_{F,V}: \frac{\bD_\dR(V)}{\Fil^0\bD_\dR(V)}\hookrightarrow \H^1(F,V)\]with image $\H^1_e(F,V)\subset\H^1(F,V)$.

The Bloch--Kato finite subspace $\H^1_f(F,V)\subset \H^1(F,V)$ is defined as\[
\H^1_f(F,V)\coloneq \ker(\H^1(F,V)\to \H^1(F,V\otimes_{\bQ_p}\bB_\cris))\]where $\bB_\cris\subset \bB_\dR$ is the ring of crystalline periods. Let $\bD_\cris(V)\coloneq(V\otimes_{\bQ_p}\bB_\cris)^{G_F}$. Then $\bD_\cris$ is a $F_0$-vector space equipped with a crystalline Frobenius action $\Phi$, where $F_0$ is the maximal unramified extension of $\bQ_p$ in $F$. Suppose $\bD_\cris(V)^{\Phi=1}=0$ where $\Phi$ is the Frobenius operator, then one could identify $\H^1_e(F,V)$ with $\H^1_f(F,V)$. Moreover, $\exp_{F,V}$ would become an isomorphism onto its image $\H^1_e(F,V)$.

If $V$ is a de Rham representation with $\bD_\cris(V)^{\Phi=1}=0$, then the Bloch--Kato logarithm is defined by the inverse of $\exp_{F,V}$\[
\log_{F,V}:\H^1_f(F,V)\xrightarrow{\isom} \frac{\bD_\dR(V)}{\Fil^0\bD_\dR(V)}\]

For our application, we will let $F$ be a finite extension of the completion of $\bQ(f)$ at a prime $\frp\mid p$ and let $V$ be the self-dual twist of $V_f$ as in the introduction. In particular, all assumptions above are satisfied and the logarithm maps extends to\[
\log_{F,V}:\H^1_f(F,V)\xrightarrow{\isom} \frac{\bD_\dR(V)}{\Fil^0\bD_\dR(V)}\iso (\Fil^1(\bD_\dR(V)))^\vee\iso F\] 
where the middle isomorphism is given by the de Rham cup product pairing\[
<,>:\bD_\dR(V)\times \bD_\dR(V)\to F\] with respect to which $\Fil^0(\bD_\dR(V))$ and $\Fil^1(\bD_\dR(V))$ are exact annihilators of each other.

One could choose a differential $\omega$ in $\Fil^1\bD_\dR(V)$, thus defining a map $\log_\omega:\H^1_f(F,V)\to F$ by composing $\log_{F,V}$ with evaluation at $\omega$.

\subsection{$p$-adic Abel--Jacobi maps}\label{AJ}
Similar to the Heegner cycles, one can study $p$-adic Abel--Jacobi maps in different settings. In this section we discuss some background, and the exact maps that are referred to as the $p$-adic Abel--Jacobi maps will be made clear in the next section.

\subsubsection{$p$-adic Abel--Jacobi map over $\Gamma(N)$}\label{pAJ/N}
Here we briefly recall the $p$-adic Abel--Jacobi map discussed in~\cite{Mas}. 

Recall the $p$-adic sheaf $\cF$  over $Y(N)$ in Section 2.1 in \text{loc. cit.} defined by \[\cF\defeq\varprojlim_n\Sym^{2r-2}(R^1\pi_*(\bZ/p^n))(r-1).\] 

Let \[J_p\defeq \Pi_B\H^1_{\text{\'e}t}(X_N\otimes \ol\bQ,j_*\cF)(r).\] 

Then the Hecke algebra $\mathbb{T}$ over $\bZ$ generated by the Hecke operators $T_\ell$ acts on $J_p$. If we write $I_f$ for the kernel of the map $\mathbb{T}\to\cO_{\bQ(f)}$ sending $T_\ell$ to $a_\ell$, one knows the continuous $G_\bQ$-representation \[A_p\defeq \{x\in J_p|I_f\cdot x=0\}
\] is $\cO$-free of rank $2$ by~\cite[Proposition 3.1]{Nek92}. In fact, $A_p\otimes F$ is identified with the self-dual twist of Deligne's representation attached to $f$, which in turn can be identified with our $V_f$ (see~\cite[Section 5.4]{Thackeray2022}). Thus we can think of $A_p\otimes \cO$ as a Galois stable lattice $T_f$ of $V_f$.

One knows that there is a map (eq. (3) in~\cite{Mas})\[
\H^1_{\textbf{\'e}t}(W'_{2r-2}\otimes\ol\bQ,\bZ_p(r))\to J_p\to A_p\]

For any number field $L$, there is a $p$-adic Abel--Jacobi map defined by (see~\cite[Section 5.6]{Thackeray2022}) \[
\Phi:\CH^r(W'_{2r-2}/L)_0\otimes_\bZ\bZ_p\to\H^1(L,\H^1_{\textbf{\'e}t}(W'_{2r-2}\otimes\ol L,\bZ_p(r)).\]

Composing $\Phi$ with the map that is $H^1(L,\cdot)$ of the above map and then applying $\otimes \cO$ or $\otimes F$ give the maps
\[
\AJ_L^f:\CH^r(W'_{2r-2}/L)_0\otimes \cO\to \H^1(L,T_f).\]
and
\[
\AJ_L^f:\CH^r(W'_{2r-2}/L)_0\otimes F\to \H^1(L,V_f).\] In particular, from~\cite[Corollary 3.2]{Mas}, one knows the images are in $\H^1_f(L,T_f)$ and $\H^1_f(L,V_f)$ respectively.

\begin{remark}
	Here in the construction, $T_f$ is naturally the `canonical lattice' in $V_f$. Start from now, we will only work with this $T_f$. Note that our setting is that the subrepresentation $\F(\phi)$ of $\ol\rho_f$ is either ramified at $p$ and even, or unramified and odd for one (and hence for all) Galois stable lattice $T$ in $V_f$ so all the results apply to this choice.
\end{remark} 

\subsubsection{$p$-adic Abel--Jacobi map over $\Gamma_1(N)$}
We next consider the $p$-adic Abel--Jacobi map for classical Heegner cycles over $\Gamma_1(N)$ defined in~\cite{BDP2017}.

Let $K$ be an imaginary quadratic field satisfying~\cref{assum}. In particular, $p=v\ol v$ is split in $K$. As in~\cite[Seciton 2.4]{BDP2017}, let $V=\H^{2r-1}((W_{2r-2})_{\ol K},\bQ_p(r))$. Let $V_f$ be the self-dual Galois representation attached to a modular form $f\in S_{2r}^{new}(\Gamma_0(N))$.

Taking $j=r$, the map $\beta_v: \CH^r(W_{2r-2}/K_1)_{0,\bQ}\to (\Fil^{r}H^{2r-1}_{dR}((W_{2r-2})_{K_1,v}))^\vee$ in~\cite[section 2.4]{BDP2017} is the $p$-adic Jacobi map that relates classical Heegner cycles over $\Gamma_1(N)$ to the BDP $p$-adic $L$-function (see~\cref{BDP} or~\cite[Theorem 4.1.3]{BDP2017}).  It is defined as a composition\[
\beta_v\defeq \PD\circ \log_{K_v,V} \circ\delta_{0,v}\]
where $\delta_{0,v}$ is the composition{\small\[
	\CH^r(W_{2r-2}/K_1)_{0,\bQ}\xrightarrow{\delta_0}\H^1(K,\H^{2r-1}((W_{2r-2})_{\ol K},\bQ_p(r))\xrightarrow{res_v}\H^1(K_v,\H^{2r-1}((W_{2r-2})_{\ol K},\bQ_p(r)),\]}
of restriction and $\delta_0=\AJ^{\'et}$ the \'etale Abel--Jacobi map, $\log_{K,V}$ is the Bloch--Kato logarithm in the previous section and $\PD$ denotes Poincar\'e Duality:
\[\PD:\frac{\bD_\dR(V)}{\Fil^0\bD_\dR(V)}=\frac{\H^{2r-1}_\dR((W_{2r-2})_{K_{1,v}})}{\Fil^r \H^{2r-1}_\dR((W_{2r-2})_{K_{1,v}}}\isom (\Fil^r\H^{2r-1}_\dR((W_{2r-2})_{K_{1,v}}))^\vee.\]

The composition makes sense because the image of $\delta_{0,v}$ is contained in the subgroup $\H^1_f(K_v,\H^{2r-1}((W_{2r-2})_{\ol K},\bQ_p(r)))$ by~\cite[Theorem 3.1(i)]{Nek00}.

We mention that the map $\delta_0$ also induces a map $\CH^r(W_{2r-2}/K_1)_{0,\bQ}\to \H^1(K_1,V_f)$ (see~\cite[Section 5.6]{Thackeray2022}).

We recall the differential $\omega_f\in\Fil^r\H^{2r-1}_\dR((W_{2r-2})_{K_{1,v}})=\Fil^{2r-1}\H^{2r-1}_\dR((W_{2r-2})_{K_{1,v}})$ associated to $f$ in~\cite[Corollary 2.3]{BDP13} (see also Lemma 2.2(3) there). The above map $\beta_v$ can be then composed with `evaluation at $\omega_f$'. From the discussion at the end of the last subsection, one can also view $\omega_f$ as in $\Fil^1\bD_\dR(V)$.

Finally, we mention that one can also make sense of the above maps with $X_{2r-2}$ in place of $W_{2r-2}$, as is the case in\cite{BDP13}. For example, one can define \[\delta^\BDP_0:
\CH^{2r-2}(X_{2r-2}/K_1)_{0,\bQ}\to\H^1(K,\H^{4r-3}((X_{2r-2})_{\ol K},\bQ_p(2r-1)),
\]and\[
\beta_v^\BDP: \CH^{2r-1}(X_{2r-2}/K_1)_{0,\bQ}\to (\Fil^{2r-1}H^{4r-3}_{dR}((X_{2r-2})_{K_{1,v}}))^\vee,
\]
where $\beta_v^\BDP$ can be composed with evaluation at $\omega_f\wedge\omega_A^{r-1}\eta_A^{r-1}\in\Fil^{2r-1}H^{4r-3}_{dR}((X_{2r-2})_{K_{1,v}})$ introduced in~\cite[Section 2.2]{BDP13}.

\subsection{Abel--Jacobi images of Heegner cycles}\label{cyclecomp}
In this subsection we focus on the applications of the results in the previous sections to our self-dual $G_K$-representation $V_f$. Recall that we are working over the canonical Galois stable lattice $T_f$ of $V_f$. Recall also that $K_1$ denotes the Hilbert class field of $K$.

For our convenience, we denote by $\AJ^f_{K_1,1}$ the map \[\delta_{0}:\CH^r(W_{2r-2}/K_1)_{0,\bQ}\otimes \cO\to \H^1(K_1,T_f).\] 
and we abbreviate the evaluation of $\beta_v=\PD\circ\log_{K_v,\H^{2r-1}((W_{2r-2}/K_1)_{\ol K},\bQ_p(r))}\circ \loc_v\circ \delta_0$ at a differential $w\in\Fil^{r}\H^{2r-1}_\dR((W_{2r-2})_{K_{1,v}})$ as $\log_w(\AJ^f_{K_1,1})$.

By abuse of notation, we also denote by $\AJ_{K_1}^f$ the map\[
\CH^r(W'_{2r-2}/K_1)_0\otimes \cO\to \H^1_f(K_1,T_f)\]

\begin{assumption}
	We assume that all $p$-adic Abel--Jacobi maps are injective.
\end{assumption}
This is a standard hypothesis in the literature. Sometimes we still call the Abel--Jacobi images of Heegner cycles `Heegner cycles'.

Notice that there is a Gross--Zagier type formula for Heegner cycles over $\Gamma(N)$ for modular forms obtained by Zhang (\cref{GZZ}). However, the $p$-adic version of BDP (\cref{BDP}) concerns the Heegner cycles over $\Gamma_1(N)$, while there is no known formula of Gross--Zagier type for $\Gamma_1(N)$. This unfortunate inconsistency is the main obstacle in obtaining a rank $1$ Tamagawa Number formula using current approaches.

Luckily, there are a few well-understood relation between the different Heegner cycles in terms of the $p$-adic Abel--Jacobi maps.

\begin{proposition}\label[proposition]{comp}
  $[\im(\AJ^f_{K_1,1}):\AJ^f_{K_1,1}(\Delta_\fra)]=[\im(\AJ^f_{K_1}):\AJ^f_{K_1}(\Delta_N)].$
\end{proposition}
\begin{proof}
	This is~\cite[Proposition 10.6, Proposition 10.7]{Thackeray2022}. In particular, one knows the index is independent of $\fra$.
\end{proof}

A consequence of this comparison is that the $\Delta_\fra$ is non-torsion if and only if $\Delta_N$ is non-torsion. 

Similarly, one can relate $\Delta_\fra$ to $\Delta_\fra^\BDP$. Define $J^{\bZ_p}_\BDP$ as in~\cite[Section 5.6]{Thackeray2022}. Then one can define $\AJ^f_{K_1,\BDP}:\CH^{2r-2}(X_{2r-2}/K_1)_0\otimes_\bZ\bZ_p\to \H^1_f(K_1,J_\BDP^{\bZ_p})$ which we also assume to be injective and one has the following relation.
\begin{proposition}\label[proposition]{compBPD}
	$[\im(\AJ^f_{K_1,1}):\AJ^f_{K_1,1}(\Delta_\fra)]=[\im(\AJ^f_{K_1,\BDP}):\AJ^f_{K_1,\BDP}(\Delta_\fra^\BDP)].$
\end{proposition}
\begin{proof}
	This is in~\cite[Section 10.5]{Thackeray2022}. Again, this index is independent of $\fra$.
\end{proof}

Consequently, $\Delta_\fra$ is non-torsion if and only if $\Delta_\fra^\BDP$ is non-torsion.

These comparisons allow us to take the advantage of the Gross--Zagier--Zhang formula for $\Delta_N$ to relate the behavior of $L$-functions to that of $\Delta_\fra^\BDP$ (or rather, $\kappa_\infty$. But see the end of~\cref{2.2}) in the Heegner point Main Conjecture (see~\cref{IMC}(IMC1)). More precisely, if one assumes $\kappa_1$ is $\Lambda_K$-nontorsion, then its projection to $\H^1_\BK(K,T_f)$, which is $\sum_{[\fra]\in\Pic(\cO_K)}\AJ^f_{K_1,\BDP}(\Delta_\fra^\BDP)$, will be $\bZ_p$-nontorsion. This implies $\Delta_\fra^\BDP$ is non-torsion an hence $\Delta_N$ is nontorsion. Finally, desptie the difference between Masoero's cycle and Zhang's cycle, it is implicit in~\cite{Mas} that\[
\AJ^f_{K_1}(\Delta_N)=\AJ^f_{K_1}(\tilde{\Gamma}).
\]In particular, if $\Delta_N$ is non-torsion, so is Zhang's cycle by~\cref{zhangcyc}. This will be the key in the proof of~\cref{pconv}.

\section{Control theorems}\label{section:control}
\subsection{A cyclotomic control theorem}
In this subsection we recall a cyclotomic control theorem for modular forms. Again let $f$ be a newform of weight $2r$. $\Lambda_\bQ\defeq\cO\llbracket \Gamma_\bQ\rrbracket$ will denote the cyclotomic Iwasawa algebra over $\bQ$, where $\Gamma_\bQ\defeq\Gal(\bQ_\infty/\bQ)$. Recall that $\fp\mid p$ is a chosen place of $\bQ(f)$. If one further assumes $a_p(f)\notequiv 1\pmod{\fp}$, the control theorem is the main result of~\cite{LV21} for $F=\bQ$. This additional hypothesis will be satisfied for our application. Indeed, by the description of the residual representation attached to a modular form (see e.g.~\cite[Theorem 34]{Kri16}), the quotient representation is given by a power of mod-$p$ cyclotomic character coming from self-dual twist multiplied by an unramified character taking $\Frob_p$ to $\alpha_p$, the unit root of the Hecke polynomial $x^2+a_p(f)+p^{2r-1}$. Now if we assume $\phi|_{G_p},\psi|_{G_p}\ne \mathbf{1},\omega$, then $\alpha_p\notequiv 1\pmod{p}$. But $a_p(f)=\alpha_p+p^{2r-1}/\alpha_1$, so $a_p(f)\notequiv 1\pmod{p}$ as well. 

Recall that $V_f=\rho_f(1-r)$ is self-dual.

\begin{definition}
	Let $L$ be an number field and let $v$ be any place of $L$. The \textit{unramified local condition} is defined as\[
	\H^1_\ur(L_v,-)=\ker\bigl(\H^1(L_v,-)\to\H^1(I_v,-)\bigr)\] where $I_v\subset G_{L_v}$ is the inertia subgroup at $v$.\\
	Let $\bB_\cris$ be Fontaine's crystalline ring of periods. If $v\mid p$, the \textit{Bloch--Kato local conditions} on $V_f$ and $A_f$ are respectively defined as\[
	\H^1_f(L_v,V_f)\defeq\ker\bigl(\H^1_f(L_v,V_f)\to\H^1(L_v,V_f\otimes_{\bQ_p}\bB_\cris)\bigr)\]
	and
	\[\H^1_f(L_v,A_f)\defeq\im\bigl(\H^1_f(L_v,V_f)\to\H^1(L_v,A_f)\bigr)\]
	where the last arrow is induced by the canonical map $\H^1(L_v,V_f)\to\H^1(L_v,A_f)$.\\
	If $v\nmid p$, the \textit{Bloch--Kato local conditions} on $V_f$ and $A_f$ are respectively defined as\[
	\H^1_f(L_v,V_f)\defeq\H^1_\ur(L_v,V_f)\]
	and
	\[\H^1_f(L_v,A_f)\defeq\im\bigl(\H^1_f(L_v,V_f)\to\H^1(L_v,A_f)\bigr)\].
	
	The \textit{Bloch--Kato Selmer group} of $A_f$ over $L$ is defined as\[
	\H^1_\BK(L,A_f)\defeq\ker\bigl(\H^1(L,A_f)\to\prod_v\frac{\H^1(L_v,A_f)}{\H^1_f(L_v,A_f)}\bigr),\]
	where $v$ runs over all places of $L$.
\end{definition}

To define Greenberg's Selmer groups, we need a new type of local conditions. Again let $-$ be $V_f$ or $A_f$. We first recall a short exact sequence\[
0\to\Fil^+(V_f)\to V_f \to \Fil^-(V_f) \to 0
\]
such that $\Fil^+(V_f)$ is one dimensional, which is characterized by the fact that $\Fil^-(V_f)$ is an unramified character times the $(1-r)$-th power of the cyclotomic character coming from the self-dual twist.
Define $\Fil^+(T_f)=T_f\cap\Fil^+(V_f)$ and let $\Fil^+(A_f)\defeq \Fil^+(V_f)/\Fil^+(T_f)$, $\Fil^-(A_f)\defeq A_f/\Fil^+(A_f)$. We mention that when $f$ is weight $2$ with associated elliptic curve $E$ of good ordinary reduction at $p$, $\Fil^+(T_pE)$ is just the kernel of the reduction map $T_pE\to T_p{\tilde{E}}$ where $\tilde{E}$ is the reduction of $E$ at $p$, and $\Fil^+(V_pE)=\Fil^+(T_pE)\otimes \bQ_p$.

Let $M_f\defeq T_f\otimes \Lambda_\bQ^\vee$ and let $-$be $V_f,\ A_f$ or $M_f$.

\begin{definition}
	The \textit{ordinary local condition} is defined as
\[\H^1_\ord(L_v,-)=\ker\bigl(\H^1(L_v,-)\to\H^1(I_v,\Fil^-(-))\bigr)\]
The \textit{Greenberg's Selmer group} is defined as\[
	\H^1_\Gr(L,M_f)\defeq \ker\bigl(\H^1(L,M_f)\to \prod_{v\mid p}\frac{\H^1(L_v,M_f)}{\H^1_\ord(L_v,M_f)}\times\prod_{v\nmid p}\frac{\H^1(L_v,M_f)}{\H^1_\ur(L_v,M_f)}\bigr)\] 
	where $v$ runs through all primes of $L$.
\end{definition}
\begin{remark}
From Shapiro's lemma, we have $\H^1(L,M_f)=\H^1(L_{\infty},A_f)$ where $L_\infty$ is the cyclotomic $\bZ_p$ extension of $L$. The same is true for the local cohomology groups and Selmer groups.
\end{remark}

By~\cite[sectoin 3.3.3]{LV21}, when $v\mid p$, one has $\H^1_f(L_v,A_f)\subset\H^1_\ord(L_v,A_f)$. When $v\nmid p$, one can show that $\H^1_f(L_v,A_f)\subset\H^1_\ur(L_v,A_f)$ and from~\cite[Lemma 3.1]{LV21}, the index $[\H^1_\ur(L_v,A_f):\H^1_f(L_v,A_f)]$ is finite. 

\begin{definition}
	Let $v$ be a place of a number field $L$ not above $p$. The \textit{$p$-part of the Tamagawa number} of $A_f$ at $v$ is the integer\[
	c_v(A_f/L):=[\H^1_\ur(L_v,A_f):\H^1_f(L_v,A_f)]\]
\end{definition}
The rest of the section is devoted to explaining the following cyclotomic control theorem. As in~\cite[section 2.2-2.3]{LV21}, let \begin{align*}
\Sigma\defeq \{\text{primes of $L$ at which $V$ is ramified}\}&\cup\{\text{primes of $L$ above $p$}\}\\&\cup\{\text{archimedean primes of $L$}\},\end{align*} which is a finite set. For the following theorem, take $L=\bQ$. Let $\bQ^\Sigma$ be the maximal extension of $\bQ$ unramified outside $\Sigma$. Then by Lemma 5.2 in \textit{op. cit.}, the Selmer groups can be redefined as\[
\H^1_\BK(\bQ,A_f)=\ker\bigl(\H^1(\bQ^\Sigma/\bQ,A_f)\to\prod_{v\in\Sigma}\frac{\H^1(\bQ_v,A_f)}{\H^1_f(\bQ_v,A_f)}\bigr)\]
and
\[\H^1_\Gr(\bQ,M_f)=\ker\bigl(\H^1(\bQ^\Sigma/\bQ,M_f)\to \frac{\H^1(\bQ_p,M_f)}{\H^1_\ord(\bQ_p,M_f)}\times\prod_{\ell\in\Sigma-\{p\}}\frac{\H^1(\bQ_\ell,M_f)}{\H^1_\ur(\bQ_\ell,M_f)}\bigr).\]

Now we can state the cyclotomic control theorem we will need.

\begin{theorem}\label{cyccontrol}
	Suppose that $\H^1_\BK(\bQ,A_f)$ is finite. Suppose the assumption at the beginning of this section are satisfied. Then\begin{enumerate}
		\item $\H^1_\Gr(\bQ,M_f)$ is $\Lambda_\bQ$-cotorsion;
		\item If $\cF$ is the characteristic power series of the Pontryagin dual of $\H^1_\Gr(\bQ,M_f)$, then $\cF(0)\ne 0$;
		\item There is an equality\[
		\#(\cO/\cF(0))=\#\H^1_\BK(\bQ,A_f)\cdot\prod_{v\in\Sigma,v\ne p}c_v(A_f/\bQ)\]
	\end{enumerate}
\end{theorem}
\begin{proof}
This is the main result of~\cite{LV21}. 
\end{proof}

\subsection{An anticyclotomic control theorem of Greenberg type}
In this section we introduce an anticyclotomic theorem similar to that in~\cite{Greenberg1999}, which will be used in the proof of the higher weight $p$-converse theorem. The notations are from~\cref{iwasawa}. We do not make the assumption from the last subsection that $a_p(f)\notequiv 1\pmod{\fp}$.

\begin{theorem}\label{anticon}
	Assume that $p\nmid 2N$. Then the map\[
	\H^1_\BK(K,A_f)\to\H^1_{\cF_{\Lambda_K}}(K,M_f)^{\Gamma_K}\]has finite kernel and cokernel.
\end{theorem}
\begin{proof}
 The proof is similar to that of~\cite[Theorem 2.4.1]{KY24b}.
\end{proof}
\subsection{An anticyclotomic control theorem of Jetchev--Skinner--Wan type}
	In this section, we consider another control theorem for the anticyclotomic Selmer groups introduced in~\cite{JSW2017}. As is in the case of~\cite{CGLS} (or rather~\cite{KY24}), the anticyclotomic Selmer groups generate the same $\Lambda_K$-characteristic ideals as the unramified Selmer groups. This control theorem is thus good for a rank $1$ Tamagawa Number formula. We remark that we do allow non-trivial global torsion in this section for future use and we do not make the assumption that $a_p(f)\notequiv 1\pmod{\fp}$.
	
	Recall that $K$ is an imaginary quadratic filed satisfying~\cref{assum}. Let $X^\Sigma_\ac(M_f)$ be the Pontryagin dual of the anticyclotomic Selmer group $\H^1_{\cF^\Sigma_\ac}(K,M_f)$ defined in~\cite{JSW2017}.

\begin{theorem}\label{control}
	Assume that
	\begin{enumerate}
		\item The $\cO_L$-module im $\AJ_K$ has rank $1$
		\item $\#\Sha_\Nek(f/K)<\infty$
		\item Localization: For each place $v\mid p$ of $K$, the localization map $\H^1_\BK(K,A_f)\to\H^1_f(K_v,A_f)$ restricts to a map\[
		(\im \AJ_K)\otimes_{\cO_L}(L/\cO_L) \to (\im AJ_{K_v})\otimes_{\cO_L} (L/\cO_L)
		\]
		of which the kernel is torsion.
		\item Local corank $1$: For each place $v\mid p$ of $K$, the $\cO_L$-module $\H^1_f(K_v,A_f)$ has corank $1$.
	\end{enumerate}
	Let $f^\Sigma_{ac}$ be a generator of the characteristic ideal $\Char_\Lambda(X^\Sigma_\ac(M_f))$ of the torsion $\Lambda$-module $X^\Sigma_\ac(M_f)$, then\begin{equation}\label{Control}
		\#\cO/f^\Sigma_{ac}(0)=\frac{\#\Sha_\BK(f/K)\cdot C^\Sigma(A_f)}{(\#\H^0(K,A_f))^2}(\#\delta_v)^2,
	\end{equation}
	
	where \[C^\Sigma(A_f)=\#\H^0(K_v,A_f)\cdot \#\H^0(K_{\ol v},A_f)\cdot\prod_{w\in S_p\setminus\Sigma, w\ split}\# \H^1_{\nr}(K_w,A_f)\cdot \prod_{w\in \Sigma}\#\H^1(K_w,A_f),\] and \[\delta_v=\frac{(\cO_L:\cO_L\cdot \log_\omega(\loc_{v}C))}{(\cO_L:\log_\omega(\H^1_f(K_v,T_f)_{/\tors}))(H^1_f(K,T_f)_{/\tors}:\cO_L \cdot C)}\]
	where $C$ is any cycle whose image under the localization map has finite index in $\H^1_f(K_v,T_f)$, and $\omega$ is any differential such that $\log_\omega$ restricts to an isomorphism $\log_{\omega}:\H^1_f(K_v,T_f)_{/\tors} \iso \cO_L$ (as a ring).

\end{theorem}
\begin{proof}
	The above formula essentially follows from the computation in~\cite[Section 3]{JSW2017}. Indeed, it is checked in~\cite[Section 8.1]{Thackeray2022} that the assumptions in~\cite{JSW2017} are satisfied, then equation~\eqref{Control} comes from~\cite[Appendix B]{KY24} (for the residually reducible case), similarly as in~\cite[Theorem 8.1]{Thackeray2022}. Note that the assumption \textit{(i) Congruence: $k/2$ is not congruent to $0$ or $1$ modulo $p-1$} from \textit{loc. cit.} is not necessary because the arguments in~\cite{KY24} do not need to assume the (HT) hypothesis from~\cite{JSW2017}. Here $\delta_v$ is the localization map\[
		\loc_{v}/\tors: \H^1_f(K,T_f)_{/\tors}\to \H^1_f(K_{v},T_f)_{/\tors},
	\]
	 and the computation of $\delta_v$ comes from that in~\cite[Theorem 8.2]{Thackeray2022}, noting that we need to replace $\H^1_f(K,T_f)$ by $\H^1_f(K,T_f)_{/\tors}$ if we allow torsion.

\end{proof}
\begin{remark}
	\begin{itemize}
		\item[(i)] We now study the assumptions in~\cref{control}. We will mostly be concerned with the hypothetical situation where one aims to get a rank $1$ Tamagawa Number formula, so assuming $\ord_{s=k}L(f,s)=1$, (i) and (ii) are natural consequences of Gross--Zagier--Zhang--Kolyvain--Nekovář theorem, where one chooses a cycle $C_N$ coming from the classical Heegner cycles considered by both Zhang and Nekovář. From the sequence~\eqref{AJdescent}, they already imply $\H^1_\BK(K,A_f)$ has corank $1$. (iv) comes from the fact that $\H^1_f(K,V_f)$ is $1$-dimensional and propagation turns rank into corank. Now (iii) is a consequence of a standard hypothesis that the localization map should be surjective or at least non-zero.
		\item[(ii)] In practice, one can take $C$ to be certain Abel--Jacobi image of Heegner cycles. One could simply take $\omega$ to be the $\omega_f$ in~\cref{AJ}. Both will be discussed in~\cref{rank1}.
	\end{itemize}
\end{remark}

\section{Proof of the $p$-part Tamagawa number conjecture formula}\label{main}
\subsection{Preliminaries}

\subsubsection{Gross--Zagier formulae}

Recall the class $S_{2r}(E_x)$ defined in~\cref{zhangcycle}. Let $V$ and $V'$ be as in~\cite[Section 0.3]{Zhang1997}, Extend $f$ to a basis $\{f=f_1,\ ...,\ f_t\}$ to an orthonormal basis of $V'$ with respect to the Petersson inner product $(\cdot,\cdot)_{\Gamma_0(N)}$ and let $V_f'$ be the $f$-eigencomponent of $V'$. Put $s_f'$ to be the image of $ S_{2r}(E_x)$ in $V_f'$ and take $\chi$ to be trivial in \textit{loc.\ cit.}. 

\begin{theorem}[Gross--Zagier--Zhang formula]\label{GZZ}

\[L'(f,r)=\frac{2^{4r-1}\pi^{2r}(f,f)_{\Gamma_0(N)}}{(2r-2)!u^2h\sqrt{\vert D\vert}}\langle s_f',s_f'\rangle.\]

\end{theorem}
\begin{proof}
	This is~\cite[Corollary 0.3.2]{Zhang1997}.
\end{proof}

Here the pairing $<,>$ is the Gillet--Soul\'e pairing, which is only conjectured to be non-degenerate. We assume it is non-degenerate, so that a Heegner cycle is non-torsion if and only if $L'(f,r)$ is nonvanishing.

\begin{assumption}
	The Gillet--Soul\'e pairing is non-degenerate.
\end{assumption}

\begin{theorem}[$p$-adic Gross--Zagier formula]\label{BDP} 
	Let $\Delta_1$ be the classical Heegner cycle over $\Gamma_1(N)$ and let $C_1=\sum_{[\fra]\in\Pic(\cO_K)} AJ^f_{K_1}(\Delta_\fra)$. Then
	\[\log_{\omega_f}(\loc_{v}C_1)^2=(-4D)^{r-1}(1-p^{-r}a_p(f)+p^{-1})^{-2}L_p(f,\mathbf{N}_K^r).\]
	Here $\omega_f$ is the differential assigned to $f$ as in~\cite[Corollary 2.3]{BDP13}. In particular, $\loc_v(C_1)$ has finite index in $\H^1_f(K_v,T)$.
\end{theorem}
\begin{proof}
	This is~\cite[Theorem 4.1.3]{BDP2017}. We make the choices $r_1=2r-2$, $j=r_2=0$ (corresponding to classical Heegner cycles) and $\chi=\mathbf{N}_K$. 
That $\loc_v(C_1)$ has finite index is an obviously corollary since the above formula shows $\cO_L\log_{\omega_f}(\loc_v(C_1))$ has finite index in $\cO_L\subset L\xleftarrow{\log_{\omega_f},\cong}\H^1_f(K_v,V)$ and $\cO_L\supset \log_{w_f}\H^1_f(K_v,T)\supset O_L\log_{w_f}(\loc_v(C_1))$. Note that $L_p(f,\mathbf{N}_K^r)$ is our notation is identified with $\cL_f^\BDP(0)$.
\end{proof}

\subsection{Computation of the local index in the Wach module}

\begin{theorem}\label{LLZ}
In~\cref{Control}, we have	\[\ord_p(\cO_L:\log_\omega(\H^1_f(K_v,T_f)_{/\tors}))=\ord_p\bigl(\frac{\#\H^0(K_w,A_f)}{1-p^{-r}a_p(f)+p^{-1}}\bigr)\]
\end{theorem}
\begin{proof}
	For brevity, we write $(V,T,A)$ for $(V_f,T_f,A_f)$.
	
	We begin the proof by noting that $\omega$ does not play any role in the formula. Indeed, by Fontaine-Laffaille theory (see for example~\cite[Theorem 6.10.8]{LLZ2014}. See also~\cite[Section 4]{BK07}), the Bloch-Kato logarithm takes $H^1_f(K_{v},T)_{/\tors}$ to $ \frac{(1-\phi)^{-1}D}{(1-\phi)^{-1}D\cap \Fil^0\mathbf{D}_\dR(V)}$, where $D\subset \mathbf{D}_\dR(V)$ is the strongly divisible lattice corresponding to $T$. Here $\phi$ is a Frobenius action.

	The map $\exp_\omega$ is the inverse of a composition of isomorphisms (see~\cref{log})
	\begin{center}\begin{tikzcd}
	\log_\omega: \H^1_f(K_v,V)\ar[r,"\log"]& \frac{\D_\dR(V)}{\Fil^0 \D_\dR(V)} \ar[r,"\isom"]& L\ar[r,"\cdot\omega"]& L\\
	\H^1_f(K_v,T)_{/\tors}\ar[r,maps to,"\log"]\ar[u, phantom, sloped, "\subset"]& \frac{(1-\phi)^{-1}D}{(1-\phi)^{-1}D\cap\Fil^0\D_\dR(V)} \ar[r,maps to,"\isom"]\ar[u, phantom, sloped, "\subset"]& p^m\cO_L\ar[u, phantom, sloped, "\subset"]\ar[r,maps to,"\cdot\omega"] & p^m\cO_L\ar[u, phantom, sloped, "\subset"]\\
	\end{tikzcd}\end{center}
where $m=\ord_p(\cO_L:\log_\omega(\H^1_f(K_v,T)_{/\tors}))$. It is then clear that we could ignore the last column and compute $m$ with the first three columns.

Now by~\cite[Theorem 4.5]{BK07}, the index we need to compute is $\ord_p(h^1(D)_{/\tors}:(1-\phi)\frac{D}{D^0})$ where $D^0\coloneq D\cap \Fil^0(\D_\dR(V))$ and $h^1(D)=\coker(1-\phi|_{D^0}: D^0\to D)$. Consider the commutative diagram
	\begin{center}\begin{tikzcd}
			&0\ar[d]&0\ar[d]&0\ar[d]&\\
		0\ar[r]&D^0\ar[r]\ar[d,"="]& D\ar[r]\ar[d,"1-\phi"]& \frac{D}{D^0}\ar[r]\ar[d,"1-\phi"] &0\\

		0\ar[r]&D^0\ar[r,"1-\phi|_{D^0}"]\ar[d]& D\ar[r]\ar[d]& h^1(D)\ar[r]\ar[d] &0\\
		
		&0&\coker(1-\phi)&\coker(\frac{D}{D^0}\xrightarrow{1-\phi} h^1(D))&\\				
\end{tikzcd}\end{center}
where $\coker(\frac{D}{D^0}\xrightarrow{1-\phi} h^1(D))$ is identified with $\coker(1-\phi)$ by snake lemma. Therefore $m=\ord_p(\frac{h^1(D)_{/\tors}}{\coker(1-\phi)})$.

We now compute the denominator using the explicit description of the strongly divisible lattices in $\mathbf{D}_\dR(V)$ and Wach modules given in~\cite[Section 2--5]{LZ2013}. Recall that we have a self-dual twist\begin{equation*}\rho_f^*(1-r)=\left(\begin{array}{cc} \chi^{r}\lambda(\alpha) & *\\ 0 & \chi^{1-r}\lambda(\alpha^{-1}) \end{array}\right)\end{equation*}
where $\chi$ is the $p$-adic cyclotomic character and $\lambda(x)$ denotes the unramified character of $G_{\bQ_p}$ mapping geometric Frobenius to $x$. Here $\alpha$ is the unit root of the Hecke polynomial \[T^2-a_p(f)T+p^{k-1}.\] 

Letting $\alpha'=p^{1-r}\alpha$, then we get the `twisted' Hecke polynomial \[T^2-p^{1-r}a_p(f)T+p\] having $\alpha'$ as a root. In particular, $\ord_p(1-\alpha')=\ord_p(1-p^{1-r}a_p+p)$.

 Now in the $(\phi,\Gamma)-$module in section 5 of \textit{op. cit.}, the matrices $P$ and $G$ giving the action of $\phi$ and a $\gamma\in\Gamma$ respectively, look like\begin{equation*}P=\left(\begin{array}{cc} \alpha & *\\ 0 & \alpha^{-1} \end{array}\right)
\end{equation*}
and\begin{equation*}
	G=\left(\begin{array}{cc} \chi(\gamma)^{r} & *\\ 0 & \chi(\gamma)^{1-r} \end{array}\right)
\end{equation*}in a basis $(v_1,v_2)$. Letting $(n_1,n_2)=(\pi^{-r}v_1,\pi^{r-1}v_2)$, then the matrices of $\phi$ and $\gamma$ in the basis $(n_1,n_2)$ are given by\begin{equation*}P'=\left(\begin{array}{cc} \frac{\pi^{r}}{\phi(\pi)^{r}}\alpha & *\\ 0 & \frac{\pi^{1-r}}{\phi(\pi)^{1-r}}\alpha^{-1} \end{array}\right)
\end{equation*}
and\begin{equation*}
G'=\left(\begin{array}{cc} \frac{\pi^r}{\gamma(\pi^r)}\chi(\gamma)^{r} & *\\ 0 & \frac{\pi^{1-r}}{\gamma(\pi^{1-r})}\chi(\gamma)^{1-r} \end{array}\right)
\end{equation*}
where $\phi(\pi)=(\pi+1)^p-1$. One checks as in \textit{op. cit.} that the $\cO\otimes \bZ_p\ldbrack \pi\rdbrack-$span of $(n_1,n_2)$ is the Wach module $\mathbb{N}(T)$. Since $D=\frac{\mathbb{N}(T)}{\pi\mathbb{N}(T)}$, the matrix of $\phi$ on $D$ looks like\begin{equation*}P''=\left(\begin{array}{cc} \frac{1}{p^{r}}\alpha & *\\ 0 & \frac{1}{p^{1-r}}\alpha^{-1} \end{array}\right),
\end{equation*}
so $\coker(1-\phi)$ is given by $\det(1-P'')=(1-p^{-r}\alpha)(1-p^{r-1}\alpha^{-1})$. Therefore,\begin{align*}
	\ord_p\bigl(\coker(1-\phi))&=\ord_p((1-p^{-r}\alpha)(1-p^{r-1}\alpha^{-1})\bigr)\\
	&=\ord_p\bigl((\frac{p^{r}-\alpha}{p^{r}})(\frac{1-p^{1-r}\alpha}{p^{1-r}\alpha})\bigr)\\
	&=\ord_p(\frac{1-\alpha'}{p})\ \ \ \ \ \ \  \text{($\alpha$ is a $p$-adic unit)}\\
	&=\ord_p(\frac{1-p^{1-r}a_p(f)+p}{p}).
\end{align*}
Finally, that $h^1(D)_{\tors}\isom \H^1_f(K_v,T)_{\tors}=\H^1(K_v,T)_{\tors}$ is identified with $\H^0(K_v,A)$ is because it's nothing but the image $\H^0(K_v,A)\to \H^1(K_v,T)$ and $\H^0(K_v,V)=0$.
\end{proof}

\subsection{Discussion of the $p$-part of Tamagawa Number formula in rank $0$}

In this section we discuss some partial results towards the $p$-part of Tamagawa Number formula for the modular form $f\in S_{2r}^{new}(\Gamma_0(N))$, and analyze what is needed to get a full proof. It is based on the cyclotomic Iwasawa Main Conjectures proved in~\cite{KY24} and the cyclotomic control theorem~\cref{cyccontrol}. 

\begin{theorem}\label{pTNC}
	Let $f\in S_{2r}^{new}(\Gamma_0(N))$ be a newform with trivial nebentypus, and let $p>2$ be a prime of good ordinary reduction for $f$. Assume that $p$ is an Eisenstein prime for $f$, i.e., the residual representation $\ol\rho_f$ is reducible, and that $2\leq 2r\leq p-1$. Also assume that the sub-representation $\F(\phi)$ of $\ol\rho_f$ is ramified at $p$ and even when restricted to the decomposition group $G_p$. 
	Further assume~\eqref{per}.
	 If $L(f,r)\ne 0$, then\[
	\ord_p(\frac{L(f,r)}{\Omega_f})=\ord_p(\#\Sha_\Nek(f/\bQ)\cdot\Tam(A_f/\bQ))\]where $\Tam(f/\bQ)=\prod_{\ell\mid N} c_\ell(A_f/\bQ)$ is the product over the bad primes $\ell$ of $f$ of the Tamagawa numbers of $f$.
\end{theorem}

\begin{proof}
	From~\cite[Theorem 4.21]{LV23}, since $L(f,r)\ne 0$, $\H^1_\BK(\bQ,A_f)$ is finite and from the sequence~\eqref{AJdescentQ}, $\im(AJ_\bQ)\otimes \bQ_p/\bZ_p=0$ and $\Sha_\Nek(f/\bQ)=\H^1_\BK(\bQ,A_f)$. In particular, $\Sha_\Nek(f/\bQ)=\Sha_\BK(f/K)=\Sha(A_f/K)[\frp^\infty]$.
	
	Let $\cF_\Gr\in\Lambda_\bQ$ be a generator of the characteristic ideal of $\H^1_\Gr(\bQ,M_f)^\vee$, then from~\cref{cyccontrol}, there is an equality\[	\#(\cO/\cF_\Gr(0))=\#\Sha(A_f/K)[\frp^\infty]\cdot\prod_{v\in\Sigma,v\ne p}c_v(A_f/\bQ).\]
	Under the given assumptions, the cyclotomic Iwasawa Main Conjecture, namely the equality\[
	(\cF_\Gr(0))=(\cL_f^\MTT(\langle\chi_\cyc\rangle^{r}))\in\Lambda_\bQ,\] 
	follows from~\cref{cycMC} (see also the end of~\cref{cycIwa}).
	
	On the other hand, from~\cref{cycpL} ,up to a $p$-adic unit (note that $k=2r\leq p-1$),\[
	\cL_f^\MTT(\langle\chi_\cyc\rangle^{r})=(1-\frac{p^{r-1}}{\alpha_p})^2\cdot\frac{L(f,r)}{\Omega_f}\]
	where $\alpha_p$ is the unit root of $x^2-a_p(f)x+p^{2r-1}$. When $r>1$, $1-\frac{p^{r-1}}{\alpha_p}$ is obviously a unit. When $r=1$, $1-\frac{1}{\alpha_p}$ is still a unit since $\alpha_p\equiv a_p(f)\notequiv 1\pmod{p}$.
	
\end{proof}

\subsection{Proof of a higher weight $p$-converse theorem}
In this section, we prove~\cref{pConvintro} in the introduction. We will follow closely the arguments in~\cite[Theorem 5.2.1]{CGLS}. As in the case for elliptic curves, the $p$-converse theorem is a consequence of the anticyclotomic Heegner Point Iwasawa Main Conjecture ((IMC1) in~\cref{IMC}), an anticyclotomic control theorem~\cref{control} as well as essentially a Kolyvagin's theorem (\cite[Theorem]{Nek92}. See also~\cite[Remark 5.4]{Vigni}).

We assume the Gillet--Soul\'e pairing is non-degenerate in this section.

We first recall Nekovář's theorem.
\begin{theorem}\label{KolyNek}
	Assume $p\nmid 2N$. Let $y_0=\cores_{K_1/K} \AJ^f_{K_1}(\Delta_N)\in\im(AJ^f_K)$ be a Heegner cycle and assume $y_0$ is non-torsion. Then\begin{enumerate}
		\item $\im(AJ^f_K)\otimes\bQ=F\cdot y_0$,
		\item $\Sha_\Nek(f/K)$ is finite. 
	\end{enumerate}
\end{theorem}
A natural consequence of this is that, in the event where $y_0$ is non-torsion, from the sequence~\eqref{AJdescent}, $\H^1_\BK(K,A_f)$ must be of corank $1$ and $\im(AJ^f_K)\otimes \bQ_p/\bZ_p$ must be its maximal divisible subgroup. Thus there is an equality\[
\Sha_\BK(f/K)=\Sha_\Nek(f/K).\]
The version of the Gross--Zagier--Zhang--Kolyvagin--Nekovář's theorem we will need is the following.
\begin{theorem}\label{GZZKN}
	Let $t\in\{0,1\}$. If $\ord_{s=r}L(f/\bQ,s)=t$, then \[\dim_F(\im(AJ^f_\bQ)\otimes \bQ)=\corank_{\bZ_p}(\H^1_\BK(\bQ,A_f))=t,\] and $\Sha_\Nek(f/\bQ)[\frp^\infty]<\infty$.
\end{theorem}
\begin{proof}
	This is a combination of~\cref{GZZ} and~\cref{KolyNek}. The proof is similar to that of~\cite[Theorem 4.21]{LV23}.
\end{proof}

We now state and prove the converse theorem.

\begin{theorem}\label{pconv}
	Let $f\in S^{new}_{2r}(\Gamma_0(N))$ be a newform of weight $2r\geq 2$ with $r$ odd and $p\nmid 2N$ be an Eisenstein prime of good ordinary reduction for $f$. Assume the Gillet--Soul\'e pairing is non-degenerate and all Abel--Jacobi maps are injective. Let $t\in\{0,1\}$. Then\[
	\corank_{\bZ_p}(\H^1_{\BK}(\bQ,A_f))=t\Rightarrow \ord_{s=r}L(f/\bQ,s)=t,
	\]
	and so $\dim_F(\im(\AJ^f_\bQ)\otimes\bQ)=t$ and $\#\Sha_\Nek(f/\bQ)[p^\infty]<\infty$.
\end{theorem}
\begin{proof}
	We will choose a suitable imaginary quadratic field $K$ where we obtain the anticyclotomic Iwasawa Main Conjectures, depending on $t\in\{0,1\}$. Let $f^K$ denote the twist of $f$ by $K$.
	
	We first assume $\corank_{\bZ_p}(\H^1_\BK(K,A_f))=1$. Choose an imaginary quadratic field $K$ such that\begin{enumerate}
		\item[(a)] $D_K<-4$ is odd,
		\item[(b)] every prime $\ell$ dividing $N$ splits in $K$,
		\item[(c)] $p$ splits in $K$, say $p=v\ol v$,
		\item[(d)] $L(f^K/\bQ,s)\ne 0$.
	\end{enumerate}
The existence of such $K$ (in fact, of an infinitude of them) is ensured by~\cite[Theorem B.1]{FH95}. Now by~\cref{GZZKN}, the last condition implies $\corank_{\bZ_p}(\H^1_\BK(\bQ,A_{f^K}))=0$ and therefore $\corank_{\bZ_p}(\H^1_\BK(K,A_f))=1$. From \cref{anticon}, this implies $\corank_{\bZ_p}(H^1_{\cF_{\Lambda_K}}(K,M_f))^{\Gamma_K})=1$. Now from~\cref{IMC}(IMC1), $(\cX_\tors)_{\Gamma_K}$ must be finite so $(\H^1_{\cF_{\Lambda_K}}(K,\bT)/\Lambda_K\cdot \kappa_\infty)_\Gamma$ must be finite as well, which implies that  $\kappa_\infty$ is non-torsion.

There is an injection $\H^1_{\cF_{\Lambda_K}}(K,\bT)_{\Gamma_K}\hookrightarrow \H^1_\BK(K,T_f)$ coming from the first cohomology of the short exact sequence\[
0\to \bT\xrightarrow{\cdot T}\bT\to T_f\to 0
.\]
It then follows that $\kappa_\infty$ and hence $\kappa_1$ has non-torsion projection in $\H^1_\BK(K,T_f)$, but by construction the projection of $\kappa_1$ is nothing but $\sum_{[\fra]\in\Pic(\cO_K)}\AJ^f_{K_1,\BDP}(\Delta_\fra^\BDP)$. By the discussion at the end of~\cref{cyclecomp}, this also means Zhang's cycle $S_{2r}(E_x)$ is non-torsion. Now~\cref{GZZ} implies $\ord_{s=r}L(f/K,s)=1$ (assuming non-degeneracy of the Gillet--Soul\'e pairing). Since $\ord_{s=r}L(f/K,s)=\ord_{s=r}L(f/\bQ,s)+\ord_{s=r}L(f^K/\bQ,s)$, it follows that $\ord_{s=r}L(f/\bQ)=1$. 

The rank $0$ case is completely analogous and we replace the condition (d) by\begin{itemize}
	\item[(d')] $\ord_{s=r}L(f^K,s)=1$.
	\end{itemize} 
The existence of infinitely many such $K$ follows from~\cite[Theorem B.2]{FH95} and by~\cref{GZZKN} again one has $\corank_{\bZ_p}(\H^1_\BK(K,A_f))=1$. By passing to Zhang's cycle and applying~\cref{GZZ}, one again gets $\ord_{s=r}L(f/K,s)=1$ which implies $L(f,r)\ne 0$.
\end{proof}

\subsection{Some discussion of the $p$-part of Tamagawa Number formula in rank $1$}\label{rank1}
Finally, we talk about some ingredients that might potentially yield a proof of the $p$-part of Tamagawa Number formula in rank $1$.

\textbf{Step 0: }We begin by recalling that, when the analytic rank is $1$, there is an identification $\Sha_\Nek(f/K)\isom\Sha_\BK(f/K)$ (see remarks after~\cref{KolyNek}). When $\Sha(f/K)[\frp^\infty]<\infty$, they are also identified with $\Sha(f/K)[\frp^\infty]$.

\textbf{Step 1: }As in the proof of~\cref{pTNC} in rank $1$ case, we choose a $K$ satisfying~\cref{assum} such that $L(f^K,r)\ne 0$. Thus $\ord_{s=r}L(f/K)=1$, and by~\cref{GZZKN} we have $\#\Sha(f/K)<\infty$. 

\textbf{Step 2: }From~\cref{IMC}(IMC2), there is a $p$-adic unit $u\in(\bZ_p^\ur)^\times$ for which\[
f^\Sigma_\ac(0)=u\cdot \cL_f^\BDP(0),\]
where $f^\Sigma_\ac$ is a generator of $\Char_{\Lambda_K}(X^\Sigma_\ac(M_f))=\Char_{\Lambda_K}(\fX_f)$.

\textbf{Step 3: }From~\cref{control}, taking $\omega=\omega_f$, there is an equality\begin{align*}
	\#\cO/f^\Sigma_\ac(0)=&\frac{\#\Sha(K,A_f)\cdot C^\Sigma(A_f)}{(\#\H^0(K,A_f))^2}\times
	\\&(\#\frac{(\cO_L:\cO_L\cdot log_{\omega_f}(\loc_{v_0}C))}{(\cO_L:\log_{\omega_f}(H^1_f(K_{v_0},T_f)/\tors))(H^1_f(K,T_f)_{/\tors}:\cO_L\cdot C)})^2.
\end{align*}

From~\cref{BDP}, there is an equality
\[\cL_f^\BDP(0)=\log_{\omega_f}(\loc_{v}C_1)^2/(-4D)^{r-1}\times (1-p^{-r}a_p(f)+p^{-1})^{2},\]
where $C_1$ is the \textit{classical Heegner point over $\Gamma_1(N)$} as in~\cref{BDP}.

From~\cref{LLZ}, taking $\omega=\omega_f$, there is an equality\[\ord_p(\cO_L:\log_\omega(\H^1_f(K_v,T)_{/\tors}))=\ord_p\bigl(\frac{\#\H^0(K_w,A_f)}{1-p^{-r}a_p(f)+p^{-1}}\bigr).\]

One would naturally hope to take $C=C_1$. Then one would get (up to a $p$-adic unit)\[
\frac{\#\Sha(K,A_f)\cdot \Tam(f/K)}{(\#\H^0(K,A_f))^2}=[\H^1_f(K,T_f)_{/\tors}:\cO_L\cdot C_1]^2.
\]
However, to understand the term on the right, a Gross--Zagier formula for $C_1$ is needed.
 
On the other hand, if we choose $C_N=\cores_{K_1/K}\AJ^f_{K_1}(\Delta_N)=\cores_{K_1/K}\AJ^f_{K_1}(\tilde{\Gamma})$ corresponding to a classical Heegner cycle over $\Gamma(N)$, then~\cref{GZZ} provides a desired description of $[\H^1_f(K,T_f)_{/\tors}:\cO_L\cdot C_N]$. However, it seems difficult to compare $[\H^1_f(K,T_f)_{/\tors}:\cO_L\cdot C_1]$ to $[\H^1_f(K,T_f)_{/\tors}:\cO_L\cdot C_N]$. One could again appeal to~\cref{comp} to relate the indices of the Heegner points in the Abel--Jacobi images. However, a direct comparison of $\im(\AJ^f_{K_1,1})$ and $\im(\AJ^f_{K_1})$ seems not easy.

\printbibliography

\end{document}